\documentclass[12pt,a4wide]{article}

\usepackage[a4paper,margin=2.5cm]{geometry}

\usepackage[usenames,dvipsnames]{xcolor}
\usepackage{amssymb,mathdots}
\usepackage{mathrsfs}
\usepackage{amsmath}
\usepackage{latexsym}
\usepackage{pstricks,pst-plot,pst-node,pst-text,pst-3d}
\usepackage{color,graphicx}
\usepackage{relsize}
\usepackage{tikz}
\usepackage{indentfirst}
\usepackage[hang]{footmisc}

\usepackage{setspace}
\def\UseBibLaTeX{1}
\usepackage[hyphens]{url}
\ifnum\UseBibLaTeX>0
\newenvironment{sloppypar*}
 {\sloppy\ignorespaces}
 {\par}
\usepackage{hyperref}

\makeatletter
\let\@fnsymbol\@arabic
\makeatother

\usepackage[style=authoryear-comp,backend=biber,uniquename=false,uniquelist=false,dashed=false,%
maxcitenames=2,maxbibnames=99]{biblatex}
\DeclareNameAlias{author}{family-given}
\DeclareRobustCommand{\varlambda}{\text{\usefont{OML}{txmi}{m}{it}\symbol{"15}}}

\makeatletter
\newcommand{\Vast}{\bBigg@{5}}
\makeatother

\newtheorem{theorem}{Theorem}[section]
\newtheorem{lemma}[theorem]{Lemma}

\newtheorem{definition}[theorem]{Definition}
\newtheorem{example}[theorem]{Example}

\newenvironment{proof}[1][Proof.]{\begin{trivlist}
\item[\hskip \labelsep {\bfseries #1}]}{\end{trivlist}}

\def\zede{\textit{\texttt{z}}}

\providecommand{\keywords}[1]
{
  \noindent\small	
  \textbf{\textsf{Keywords}} #1
}
\providecommand{\mathsubclass}[1]
{
  \noindent\small	
  \textbf{\textsf{Mathematics Subject Classification}} #1
}

\title{Invariantised Euler-Lagrange equations and conserved quantities for nonconservative Herglotz variational problems}
\author{T\^{a}nia M. N. Gon\c{c}alves\footnote{\hspace{0.1cm} \href{mailto:t.m.n.goncalves@ufcat.edu.br}{t.m.n.goncalves@ufcat.edu.br}, Institute of Mathematics and Technology, Federal University of Catalao, 75704-020 Catalao -- GO, Brazil}$^{\;\, ,2}$ \and Delfim F. M. Torres\footnote{\hspace{0.1cm} Center for Research and Development in Mathematics and Applications (CIDMA), 
Department of Mathematics, University of Aveiro, 3810-193 Aveiro, Portugal}
 \and Gast\~ao S. F. Frederico\footnote{\hspace{0.1cm} Federal University of Ceara, 62900-420 Russas -- CE, Brazil}}
\date{}

\begin{document}
\maketitle

\begin{abstract}
\noindent In this paper the structures of the generalised Euler-Lagrange equations and their associated conserved quantities are derived for one-dimensional Herglotz variational problems of order $n$. Their derivations use the framework of moving frames and invariant calculus of variations. The knowledge of these structures not only offers a geometric insight, it may provide a more efficient path for the determination of extremals. This is exemplified with a Herglotz problem invariant under the restricted Lorentz group $SO^+(1,2)$.
\end{abstract}


\vspace{0.4cm}
\mathsubclass{53A55 $\cdot$ 58E30 $\cdot$ 70H03 $\cdot$ 49S05 $\cdot$ 70G65 $\cdot$ 70H33}

\section{Introduction}

Many physical systems described by differential equations arise from the classical variational principle, which selects the path or state that extremises the action. However, numerous physical processes fall outside this framework, particularly nonconservative processes. For instance, the motion of an accelerated charged particle emits electromagnetic radiation whose associated self‑force (the radiation reaction) cannot be captured by the classical principle. Likewise, an object moving through a resistive medium --- such as a sphere falling through a viscous fluid --- experiences drag forces that continuously dissipate mechanical energy into heat.

Dissipative effects motivated Herglotz, in the 1930s, to generalise the classical variational principle by replacing the action integral with a dynamical variable defined through a differential equation \parencite{GuentherGuentherGottschHerglotz}. \textcite{GeorgievaGuenther2002} subsequently established a Noether‑type theorem for the corresponding one‑dimensional first‑order Lagrangians. Later, \textcite{SantosMartinsTorres2014,SantosMartinsTorres2015} extended this framework to higher‑order one‑dimensional problems, deriving the generalised Euler–Lagrange equations and their associated conservation laws.

Although this coordinate-based formulation provides a clear and practical description of the generalised variational problem, it does not reveal the geometric mechanisms underlying the Herglotz principle. Over the last decade, \textcite{BravettiCruzTapias2017}, and \textcite{LeonSardon2017} pioneered the use of modern contact geometry for dissipative systems. Building on this, several works developed complementary geometric formulations of the Herglotz variational principle for first‑order Lagrangians on 1-contact manifolds with one independent variable (time). These include the contact formulation of singular systems via precontact geometry \parencite{LeonLainz2019}, the development of nonholonomic \parencite{LeonJimenezLainz2021} and vakonomic Herglotz dynamics \parencite{LeonLainzMunoz2021}, and the interpretation of Herglotz optimal control problems as contact vakonomic systems \parencite{LeonLainzMunoz2023}. These papers, together with the foundational analysis of contact Lagrangian systems in \textcite{GasetetAL2020}, provide a coherent geometric framework for Herglotz-type variational problems.

Further research in this coordinate-free setting includes the extension to nonsmooth trajectories \parencite{LopezColomboLeon2023}, the study of symmetry and thermodynamic consequences \parencite{BravettiGarcia2021,BravettiGarciaTapias2023}, and applications of contact geometry to dissipative, structural, or algebraic contexts \parencite{VijayanetAL2026,LiuTorres2022,SimoesetAL2025,CarinenaetAL2024}. Higher‑order contact Lagrangian systems have also been considered in \textcite{LeonetAL2021}.

While the literature on geometric formulations of Herglotz‑type systems is substantial, it does not yet provide a structural theory for the Herglotz Euler-Lagrange equations and their associated conserved quantities. The present work fills this gap by establishing both, within a fully intrinsic, coordinate‑free framework. An analogous structural theory is already available for the classical variational principle \parencite{GoncalvesMansfield2011}, and the results obtained here may be viewed as its natural extension to the Herglotz setting. Taken together, these observations motivate the structural and geometric analysis developed in this paper.

In what follows, the main results of this paper identify the geometric and algebraic structures underlying the generalised Euler-Lagrange equations and their associated conserved quantities for one-dimensional higher-order Herglotz variational problems. These structures not only provide a geometric understanding of the problem but may also yield a systematic and efficient method for solving generalised variational problems, as illustrated in the following motivating example.

\subsection*{Motivating example}
Consider the following generalised variational problem
\begin{equation}\label{dampedVar}
\frac{\mathrm{d}A}{\mathrm{d}\lambda}=\frac{1}{2}\left(t_\lambda^2-x_\lambda^2\right)-\alpha A,
\end{equation}
where $\alpha$ is a constant. This differential equation is invariant under the following group action
\begin{equation}\label{groupAction}
\widetilde{\lambda}=\lambda,\quad
\begin{pmatrix}\widetilde{t}\\
\widetilde{x}\end{pmatrix}=\begin{pmatrix}
\cosh{\theta} & -\sinh{\theta}\\
-\sinh{\theta} & \cosh{\theta}
\end{pmatrix}\begin{pmatrix}
t\\x
\end{pmatrix}+\begin{pmatrix}
a\\b
\end{pmatrix},\quad \widetilde{A}=A,
\end{equation}
where the matrix represents a boost in the $x$-direction and the vector $( a \quad b)^T$, a translation in time and space. 

By Herglotz results, the generalised Euler-Lagrange equations of \eqref{dampedVar} are
$$\left\{\begin{array}{l}
-t_{\lambda\lambda}-\alpha t_\lambda=0,\\
x_{\lambda\lambda}+\alpha x_\lambda=0.
\end{array}\right.$$
Since the variational problem \eqref{dampedVar} is invariant under the group action \eqref{groupAction}, so is its generalised Euler-Lagrange system, which can be written as
\begin{align}\label{gEL_t}
&\mathsf{E}^t(L) = -\frac{\mathrm{d}}{\mathrm{d}s}\left(\mathrm{e}^{\,\alpha \lambda}\eta\right)=0,\\\label{gEL_x}
&\mathsf{E}^x(L) = \mathrm{e}^{\,\alpha\lambda}\kappa\eta^2=0,
\end{align} 
where 
$$\eta=\sqrt{t_\lambda^2-x_\lambda^2}\quad \textrm{and}\quad \kappa=\frac{t_\lambda x_{\lambda\lambda}-x_\lambda t_{\lambda\lambda}}{(t_\lambda^2-x_\lambda^2)^{3/2}}$$
are invariants known as the \textit{Minkowski norm} and the \textit{Minkowski curvature}, respectively. So from this formulation we obtain immediately that the curvature is zero, i.e. the extremal is a straight line in spacetime, and the parametrisation is exponentially damped.

Furthermore, the conserved quantities from the Noether-type theorem can be written in the matrix form 
$\mathcal{F}(\lambda)A(t,x,t_\lambda,x_\lambda)\boldsymbol{\upsilon}=\mathbf{c}$, where $\mathcal{F}(\lambda)$ is an integrating factor, $\boldsymbol{\upsilon}$ is a vector of invariants of the group action and $\mathbf{c}=(c_1,c_2,c_3)^T$
is a vector of constants of integration, as shown below
$$\mathrm{e}^{\,\alpha \lambda}\begin{pmatrix}
\dfrac{t_\lambda}{\eta} & -\dfrac{x_\lambda}{\eta} & 0\\
-\dfrac{x_\lambda}{\eta} & \dfrac{t_\lambda}{\eta} & 0\\
\dfrac{tx_\lambda-xt_\lambda}{\eta} & \dfrac{xx_\lambda-tt_\lambda}{\eta} & 1
\end{pmatrix}\begin{pmatrix}
\eta\\
0\\
0
\end{pmatrix}=\begin{pmatrix}
c_1\\
c_2\\
c_3
\end{pmatrix},
$$
where the first conserved quantity comes from translation in time, the second from translation in space and the third from a boost in the $x$-direction.

Simplifying simultaneously these generalised conservation laws yields the following three equations in terms of $t$, $x$, $t_\lambda$ and $x_\lambda$:
\begin{align}\nonumber
&\dfrac{t_\lambda}{\eta}c_1 + \dfrac{x_\lambda}{\eta}c_2 = \mathrm{e}^{\,\alpha \lambda}\eta,\\\nonumber
&\dfrac{x_\lambda}{\eta}c_1 + \dfrac{t_\lambda}{\eta}c_2 = 0\\\nonumber 
& xc_1 + tc_2 +c_3 = 0.
\end{align}
These equations are generalised first integrals of the generalised Euler-Lagrange equations.

Obviously, the system is not difficult to solve in its original variables: the solution is 
\begin{align*}
&t(\lambda)=c_1+c_2\mathrm{e}^{\,-\alpha \lambda},\\
&x(\lambda)=c_3+c_4\mathrm{e}^{\,-\alpha \lambda}.
\end{align*}
Calculating the Minkowski norm and curvature, we obtain exactly the same results as those provided by the Equations \eqref{gEL_t} and \eqref{gEL_x}. The advantage of the use of the structured generalised Euler-Lagrange equations and conserved quantities is that it reduces the order of the problem.

In order to show the structures behind the Euler-Lagrange equations and the conserved quantities of a one-dimensional generalised variational problem of order $n$, we will need to present a brief review of concepts necessary to understand our results. These concepts range from moving frames, as formulated by \textcite{FelsOlver1998,FelsOlver1999}, to symbolic invariant calculus \parencite{Mansfield2010}.

\section{Moving frames, Adjoint action and symbolic invariant calculus}
Finding the solution to a system of differential equations can be a daunting task. However, when this system comes from a variational principle that is invariant under a prolonged group action, it is possible to pull it  back to a cross section where its solution is simpler to obtain. What allows this pull back is the notion of moving frame. We will briefly discuss this construction in this section, along with concepts in symbolic invariant calculus.

\subsection{Moving frame}\label{subsecMF}
Let $G$ be a Lie group acting smoothly on $J^n$, the $n^{\textrm{th}}$ jet space, i.e. 
$$\begin{array}{rcl}
G\times J^n & \rightarrow & J^n\\
(g,\zede) & \mapsto & \widetilde{\zede} = g\cdot \zede,
\end{array}$$
where $\zede=\left(x,u^{(n)}\right)$, with $x=(x_1,\dots,x_p)$ denoting the independent variables and $u^{(n)}$ the dependent variables and their derivatives up to order $n$.

Furthermore, suppose this action is free and regular in some domain $\mathcal{U}\subset J^n$. Under these conditions, an equivariant map $\rho: \mathcal{U}\rightarrow G$ can be defined such that 
$$\rho(\zede)\cdot \zede =k,$$
where $k$ is the unique element in the intersection of the orbit of $\zede$ and the cross section $\mathcal{K}$, which intersects transversally and uniquely the orbits of $\mathcal{U}$. This equivariant map is known as the \textit{right moving frame} relative to the cross section $\mathcal{K}$, whose coordinates are invariants under $G$. In other words, the moving frame ``normalises'' the points in $\mathcal{U}$ by moving them to the cross section $\mathcal{K}$.

The cross section $\mathcal{K}$ is defined by the so-called \textit{normalisation equations}, which are usually chosen to simplify the computations: thus, $\mathcal{K}$ is not unique. Let these normalisation equations be represented by
$$\psi_i(\widetilde{\zede})=0, \quad i=1,\dots,r,$$
where $r$ is the dimension of the Lie group $G$. Solving these for the group parameters, $(\varepsilon_1,\dots,\varepsilon_r)$, gives us the moving frame in parametric form.

\begin{example}\label{exFrame}
Consider the group action \eqref{groupAction} in our motivating example. Since that group depends on three parameters, we need three normalisation equations, which we choose to be $\widetilde{t}=0$, $\widetilde{x}=0$ and $\widetilde{x_\lambda}=0$. Note that $\widetilde{x_\lambda}$ is the induced action on $x_\lambda$ which is obtained via the chain rule,
$$\widetilde{x_\lambda}=\widetilde{x}_{\widetilde{\lambda}}=\frac{\mathrm{d}\widetilde{x}}{\mathrm{d}\widetilde{\lambda}}=\frac{1}{\mathrm{d}\widetilde{\lambda}/\mathrm{d}\lambda}\frac{\mathrm{d}\widetilde{x}}{\mathrm{d}\lambda}=-\sinh{\theta}\,t_\lambda+\cosh{\theta}\,x_\lambda.$$

Solving the normalisation equations for the group parameters yields
\begin{equation}\label{frame_param}
\theta = \ln\left(\sqrt{\frac{t_\lambda+x_\lambda}{t_\lambda-x_\lambda}}\right),\quad a = \frac{xx_\lambda-tt_\lambda}{\sqrt{t_\lambda^2-x_\lambda^2}},\quad b = \frac{tx_\lambda-xt_\lambda}{\sqrt{t_\lambda^2-x_\lambda^2}},
\end{equation}
the frame in parametric form, which can be written as
$$\rho(t,x,t_\lambda,x_\lambda)=\begin{pmatrix}
\dfrac{t_\lambda}{\eta} & -\dfrac{x_\lambda}{\eta} & \dfrac{xx_\lambda-tt_\lambda}{\eta}\\
-\dfrac{x_\lambda}{\eta} & \dfrac{t_\lambda}{\eta} & \dfrac{tx_\lambda-xt_\lambda}{\eta}\\
0 & 0 & 1
\end{pmatrix}$$
in matrix form, where $\eta =\sqrt{t_\lambda^2-x_\lambda^2}$.
\end{example}

\begin{theorem}
If $\rho$ is a right moving frame, then 
$$\left.\widetilde{\zede}\right|_{g=\rho(\zede)}=\left. g\cdot \zede\right|_{g=\rho(\zede)}=\rho(\zede)\cdot \zede$$
is invariant under the induced group action of $G$ on $J^n$. 
\end{theorem}

Using the definitions of left and right actions, it is not difficult to prove the latter result \parencite{Mansfield2010}.

For any prolonged group action on the jet space, the invariantised jet coordinates are represented as 
$$I^i=I(x_i)=\left.\widetilde{x_i}\right|_{g=\rho(\zede)},\quad I^\alpha=I(u^\alpha)=\left.\widetilde{u^\alpha}\right|_{g=\rho(\zede)},\quad I^\alpha_K=I(u^\alpha_K)=\left.\widetilde{u^\alpha_K}\right|_{g=\rho(\zede)},$$
where $i=1,\dots,p$, $\alpha=1,\dots,q$ and $K$ are multi-indices of differentiation. These are known as the \textit{normalised differential invariants}.

\begin{example}
Considering again the induced group action on $J^n$, defined in the motivating example, we calculate here some normalised differential invariants, namely,
\begin{align*}
&I^t=\left.\widetilde{t}\right|_{g=\rho(\zede)}=0, &
I^x=\left.\widetilde{x}\right|_{g=\rho(\zede)}=0,\\
&I^t_1=\left.\widetilde{t_\lambda}\right|_{g=\rho(\zede)}=\sqrt{t_\lambda^2-x_\lambda^2}, &I^x_1=\left.\widetilde{x_\lambda}\right|_{g=\rho(\zede)}=0,\\
&I^x_{11}=\left.\widetilde{x_{\lambda\lambda}}\right|_{g=\rho(\zede)}=\dfrac{t_\lambda x_{\lambda\lambda}-x_\lambda t_{\lambda\lambda}}{\sqrt{t_\lambda^2-x_\lambda^2}}. & 
\end{align*}
As expected $I^t$, $I^x$ and $I^x_1$ yield 0, since $\rho(\zede)$ is the solution to the normalisation equations. The invariants $I^t_1$ and $I^x_{11}$ are the lowest order differential invariants, from which all other invariants can be obtained. The well-known Minkowski curvature
$$\kappa = \dfrac{t_\lambda x_{\lambda\lambda}-x_\lambda t_{\lambda\lambda}}{(t_\lambda^2-x_\lambda^2)^{3/2}} \quad\textrm{ is equal to }\quad \frac{I^x_{11}}{(I^t_1)^2}.$$
\end{example}

It is possible to express the Minkowski curvature in terms of the generating differential invariants, $I^t_1$ and $I^x_{11}$, without solving for the frame. For that we use the generalisation of the Replacement Theorem, due to  \textcite{FelsOlver1999}, which allows to substitute all variables in an invariant function by its corresponding invariantised jet coordinates.

\subsection{Invariant differentiation}
In a similar way as for the normalised differential invariants, we can define the \textit{invariant differential operators} as
$$\mathcal{D}_i=\left.\widetilde{D_i}\right|_{g=\rho(\zede)},$$
where
$$\widetilde{D_i}=\frac{\mathrm{d}}{\mathrm{d}\widetilde{x_i}}=\sum_{j=1}^p\left(J^{-T}\right)_{ij}\frac{\mathrm{d}}{\mathrm{d}x_j}=\sum_{j=1}^p\left(J^{-T}\right)_{ij}D_j,$$
with $J=\dfrac{\mathrm{d}(\widetilde{x_1},\dots,\widetilde{x_p})}{\mathrm{d}(x_1,\dots,x_p)}$.

\begin{example}\label{badDiff}
Continuing with the motivating example, we know that 
$$I^x_{11}=\left.\widetilde{x_{\lambda\lambda}}\right|_{g=\rho(\zede)}=\frac{t_\lambda x_{\lambda\lambda}-x_\lambda t_{\lambda\lambda}}{\sqrt{t_\lambda^2-x_\lambda^2}}.$$

Calculating $\mathcal{D}_\lambda I^x_{11}$, we obtain
$$\mathcal{D}_\lambda I^x_{11}=\widetilde{D_\lambda}|_{g=\rho(\zede)}\,\widetilde{x_{\lambda\lambda}}|_{g=\rho(\zede)}=D_\lambda\left(\frac{t_\lambda x_{\lambda\lambda}-x_\lambda t_{\lambda\lambda}}{\sqrt{t_\lambda^2-x_\lambda^2}}\right).$$

We know that
$$\begin{array}{rl}
I^x_{111}&=\widetilde{x_{\lambda\lambda\lambda}}|_{g=\rho(\zede)}=\left.\left(\widetilde{D_\lambda}\widetilde{x_{\lambda\lambda}}\right)\right|_{g=\rho(\zede)}=\left(D_\lambda(-\sinh{\theta}\,t_{\lambda\lambda}+\cosh{\theta}\,x_{\lambda\lambda})\right)|_{g=\rho(\zede)}\\[10pt]
&=\left.(-\sinh{\theta}\,t_{\lambda\lambda\lambda}+\cosh{\theta}\,x_{\lambda\lambda\lambda})\right|_{g=\rho(\zede)}=\dfrac{t_\lambda x_{\lambda\lambda\lambda}-x_\lambda t_{\lambda\lambda\lambda}}{\sqrt{t_\lambda^2-x_\lambda^2}}.
\end{array}$$
Clearly, $\mathcal{D}_\lambda I^x_{11}\ne I^x_{111}$.
\end{example}

Thus, from Example \ref{badDiff}, we see that, in general, $\mathcal{D}_iI^\alpha_K\ne I^\alpha_{Ki}$, as opposed to
$$\frac{\mathrm{d}}{\mathrm{d}x_i}u^\alpha_K=u^\alpha_{Ki},$$
where $u^\alpha$ is \textit{bona fide}. In fact,
$$\mathcal{D}_iI^\alpha_K= I^\alpha_{Ki}+M^\alpha_{Ki},$$
where $M^\alpha_{Ki}$ is known as the \textit{error term}. This informs us that the processes of differentiation and invariantisation do not commute. Software packages for computing these error terms are available; they rely on the normalisation equations and the infinitesimal generators of the group action (introduced in the next section). A prominent example is AIDA, developed by \textcite{Hubert2007}.

\begin{example}\label{exSys}
Carrying on with the motivating example, now suppose that $t=t(\lambda,\nu)$ and $x=x(\lambda,\nu)$, and that $\nu$ is invariant under the group action. Since the normalisation equations are $\widetilde{t}=0$, $\widetilde{x}=0$ and $\widetilde{x_\lambda}=0$, the generating differential invariants are $I^t_1$, $I^t_2$, $I^x_{11}$ and $I^x_2$. As shown in Figure \ref{illustration}, there are two paths of differentiation that can be taken to reach $I^t_{12}$ and two for $I^x_{112}$.

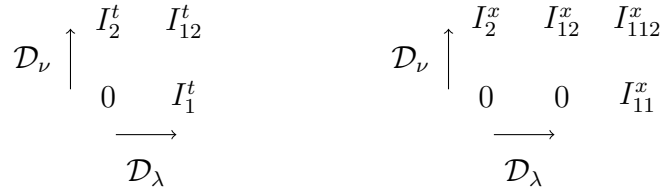
\begin{figure}[h]
\centering
\caption{Diagram of invariants}\label{illustration}
\vspace{0.2cm}
\begin{tikzpicture}
\node (0) at (0,0) {$0$};
\node (1Inv) at (1,0) {$I^t_1$};
\node (2Inv) at (0,1) {$I^t_2$};
\node (12Inv) at (1,1) {$I^t_{12}$};
\draw[->] (0.1,-0.5) -- (0.9,-0.5);
\draw[->] (-0.5,0.1) -- (-0.5,0.9);
\node (Ds) at (0.5,-1) {$\mathcal{D}_\lambda$};
\node (Dtau) at (-1,0.5) {$\mathcal{D}_\nu$};
\node (0) at (5,0) {$0$};
\node (10) at (6,0) {$0$};
\node (11Inv) at (7,0) {$I^x_{11}$};
\node (2Inv) at (5,1) {$I^x_2$};
\node (12Inv) at (6,1) {$I^x_{12}$};
\node (112Inv) at (7,1) {$I^x_{112}$};
\draw[->] (5.1,-0.5) -- (5.9,-0.5);
\draw[->] (4.5,0.1) -- (4.5,0.9);
\node (Ds) at (5.5,-1) {$\mathcal{D}_\lambda$};
\node (Dtau) at (4,0.5) {$\mathcal{D}_\nu$};
\end{tikzpicture}
\end{figure}

Differentiating $I^t_1$ with respect to $\nu$ and $I^t_2$  with respect to $\lambda$ we obtain, respectively,
\begin{align}\label{diffIt1}
& \mathcal{D}_\nu I^t_1 = I^t_{12},\\\label{diffIt2}
& \mathcal{D}_\lambda I^t_2 = I^t_{12} -\dfrac{I^x_{11}I^x_2}{I^t_1}.
\end{align}
From \eqref{diffIt1} and \eqref{diffIt2} stems the relation
\begin{equation}\label{SysMotivating1}
\mathcal{D}_\nu I^t_1=\mathcal{D}_\lambda I^t_2+\dfrac{I^x_{11}I^x_2}{I^t_1}.
\end{equation}

Similarly, we have two paths of differentiation from $I^x_{112}$, namely
\begin{align*}
& \mathcal{D}_\nu I^x_{11} = I^x_{112} -\dfrac{I^t_{11}I^x_{12}}{I^t_1},\\
& \mathcal{D}_\lambda^2 I^x_2 = I^x_{112}-\dfrac{2I^x_{11}I^t_{12}+I^x_{111}I^t_2}{I^t_1}+\dfrac{2I^x_{11}I^t_{11}I^t_2+(I^x_{11})^2I^x_2}{(I^t_1)^2},
\end{align*}
which produce the differential identity
\begin{equation}\label{SysMotivating2}
\mathcal{D}_\nu I^x_{11}+\dfrac{I^t_{11}I^x_{12}}{I^t_1} = \mathcal{D}_\lambda^2 I^x_2 +\dfrac{2I^x_{11}I^t_{12}+I^x_{111}I^t_2}{I^t_1}-\dfrac{2I^x_{11}I^t_{11}I^t_2+(I^x_{11})^2I^x_2}{(I^t_1)^2}.
\end{equation}
\end{example}

From Example \ref{exSys}, we can observe that if the multi-indices of differentiation $J$, $K$, $L$ and $M$ are such that $JK=LM$, then $I^\alpha_{JK}=I^\alpha_{LM}$, which in turn implies that
\begin{equation}\label{syzygies}
\mathcal{D}_KI^\alpha_J-M^\alpha_{JK}=\mathcal{D}_MI^\alpha_L-M^\alpha_{LM}.
\end{equation}
Differential relations such as \eqref{syzygies} are also known as \textit{syzygies}. These will play an important role in our results.

Finally we present the last concept crucial to the structure of the generalised conserved quantities, proved in section \ref{IGVP}: the Adjoint action of a Lie group $G$ on its Lie algebra.

\subsection{Adjoint action}\label{AdjointRep}

Suppose that the Lie group $G$, parametrised by $\boldsymbol{\varepsilon}=(\varepsilon_1,\dots,\varepsilon_r)$, acts on the base space $(x,u)$, with $x=(x_1,\dots,x_p)$ and $u=(u^1,\dots,u^q)$, as follows
\begin{equation}\label{LieGroupAction}
\begin{array}{ll}
\widetilde{x_i}=\Xi^i(x,u;\boldsymbol{\varepsilon}), & i=1,\dots,p,\\
\widetilde{u^\alpha}=\Phi^\alpha(x,u;\boldsymbol{\varepsilon}), & \alpha=1,\dots,q.
\end{array}
\end{equation}

The associated Lie algebra, $\mathfrak{g}$, has as basis
$$\mathbf{v}_k=\left.\left(\sum_{i=1}^p\frac{\partial \Xi^i}{\partial \varepsilon_k}\frac{\partial}{\partial x_i}+\sum_{\alpha=1}^q\frac{\partial \Phi^\alpha}{\partial \varepsilon_k}\frac{\partial}{\partial u^\alpha}\right)\right|_{g=e}=\sum_{i=1}^p\xi^i_k\frac{\partial}{\partial x_i}+\sum_{\alpha=1}^q\phi^{\alpha}_{,k}\frac{\partial}{\partial u^\alpha}, \quad k=1,\dots ,r,$$
where $e$ is the identity element in $G$. These basis elements are known as the \textit{infinitesimal generators} of the group action.

The transformation defined in \eqref{LieGroupAction} induces an action on the jet space $J^n$. This action is known as the \textit{prolonged group action}. The Lie algebra of this prolonged action, $\mathcal{X}_G(J^n)$, is generated by 
$$\mathbf{w}_k=\mathsf{pr}^{(n)}\mathbf{v}_k=\mathbf{v}_k+\sum_{\alpha=1}^q\sum_K \phi^{\alpha}_{K,k}\frac{\partial}{\partial u^\alpha_K}, \quad k=1,\dots,r,$$
where the terms $\phi^{\alpha}_{K,k}$, with $1\le |K|\le n$, are
$$\phi^{\alpha}_{K,k}=D_K\left(\phi^{\alpha}_{,k}-\sum_{i=1}^p\xi^i_ku^\alpha_i\right)-\sum_{i=1}^p\xi^i_ku^\alpha_{K,i},\quad k=1,\dots,r,$$
with $u^\alpha_{K,i}=\partial u^\alpha_K/\partial x_i$. These vector fields are known as the \textit{prolonged infinitesimal generators}.

For simplicity of the formulae to be presented shortly, let $\zede$, as defined in the beginning of section \ref{subsecMF}, be written as $\zede=(\zede_1,\dots,\zede_\ell)$, and thus, the prolonged infinitesimal generators become
\begin{equation}\label{ProlongIG}
\mathbf{w}_i=\sum_{j=1}^\ell \left.\frac{\partial \widetilde{\zede_j}}{\partial \varepsilon_i}\right|_{g=e}\frac{\partial}{\partial \zede_j}=\sum_{j=1}^\ell \omega^j_i\frac{\partial}{\partial\zede_j}=\boldsymbol{\omega}_i(\zede)^T\nabla, \quad i=1,\dots,r.
\end{equation}

\begin{definition}
Let $G$ be a Lie group and $\mathcal{X}(J^n)$ the space of vector fields on $J^n$. The Adjoint action of $G$ on $\mathcal{X}(J^n)$ is defined as
$$
\begin{array}{rrcl}
Ad: & G\times \mathcal{X}(J^n) & \rightarrow & \mathcal{X}(J^n)\\
& (g,\mathbf{v}) & \mapsto & Ad_g(\mathbf{v})=\displaystyle{\sum_{i=1}^\ell f_i(\widetilde{\zede})\frac{\partial}{\partial \widetilde{z_i}}=\mathbf{f}(\widetilde{\zede})^T\widetilde{\nabla}},
\end{array}
$$
where 
$$\mathbf{v}=\sum_{j=1}^\ell f_j(\zede)\frac{\partial}{\partial\zede_j}.$$
This definition can be simplified to
\begin{equation}\label{adjointVector}
Ad_g(\mathbf{v})=\mathbf{f}(\widetilde{\zede})^T\frac{\partial(\widetilde{\zede_1},\dots,\widetilde{\zede_\ell})}{\partial(\zede_1,\dots,\zede_\ell)}^{-T}\nabla,
\end{equation}
where $\dfrac{\partial(\widetilde{\zede_1},\dots,\widetilde{\zede_\ell})}{\partial(\zede_1,\dots,\zede_\ell)}$ is the Jacobian of the map $\zede\mapsto g\cdot \zede=\widetilde{\zede}$.
\end{definition} 

Since $Ad_g$ is a linear map and $Ad_g(\mathbf{v})\in \mathcal{X}_G(J^n)$, for all $\mathbf{v}\in \mathcal{X}_G(J^n)$, this implies that the Adjoint action of $g\in G$ on a linear combination of the basis elements of $\mathcal{X}_G(J^n)$ can be written as linear combination of \eqref{ProlongIG}, as follows
\begin{equation}\label{adjointCoefficients}
Ad_g\left(\sum_{i=1}^r\alpha_i \mathbf{w}_i \right)=\sum_{i=1}^r \alpha_i Ad_g\left(\mathbf{w}_i\right)=\sum_{j=1}^r \sum_{i=1}^r \alpha_i \left(\mathcal{A}d(g)\right)_{ij}\mathbf{w}_j.
\end{equation}
Hence, this defines the matrix $\mathcal{A}d(g)$, which is a matrix representation of $G$, called the \textit{Adjoint representation of} $G$.
Equation \eqref{adjointCoefficients} can be written in matrix form as
\begin{equation}\label{LHSresult}
Ad_g\left(\sum_{i=1}^r\alpha_i\mathbf{w}_i\right)=\boldsymbol{\alpha}\,\mathcal{A}d(g)
\Omega(\zede)\nabla,
\end{equation}
where 
\begin{equation}\label{MatrixInfs}
\Omega(\zede)=\begin{pmatrix}
\boldsymbol{\omega}_1(\zede)^T\\
\vdots\\
\boldsymbol{\omega}_r(\zede)^T
\end{pmatrix}
\end{equation}
is the matrix of infinitesimals (see \eqref{ProlongIG}).

However, from \eqref{adjointVector} we also know that 
\begin{equation}\label{adjointTrans}
Ad_g\left(\sum_{i=1}^r\alpha_i\,\mathbf{w}_i\right)=\sum_{i=1}^r \alpha_i\,\boldsymbol{\omega}_i(\widetilde{\zede})^T\dfrac{\partial(\widetilde{\zede_1},\dots,\widetilde{\zede_\ell})}{\partial(\zede_1,\dots,\zede_\ell)}^{-T}
\nabla.
\end{equation}
As for Equation \eqref{adjointCoefficients}, we can represent Equation \eqref{adjointTrans} in matrix form as
\begin{equation}\label{RHSresult}
Ad_g\left(\sum_{i=1}^r\alpha_i\mathbf{w}_i\right)=\boldsymbol{\alpha}\,\Omega(\widetilde{\zede})\dfrac{\partial(\widetilde{\zede_1},\dots,\widetilde{\zede_\ell})}{\partial(\zede_1,\dots,\zede_\ell)}^{-T}\nabla.
\end{equation}

Observing \eqref{LHSresult} and \eqref{RHSresult}, we can conclude that
\begin{equation}\label{essentialPiece}
\mathcal{A}d(g)\Omega(\zede)=\Omega(\widetilde{\zede})\dfrac{\partial(\widetilde{\zede_1},\dots,\widetilde{\zede_\ell})}{\partial(\zede_1,\dots,\zede_\ell)}^{-T}.
\end{equation}
This equality is essential for the proof of our results, but before we proceed, we shall determine the Adjoint representation of the $(1+1)$ Poincaré group present in our motivational example.

\begin{example}
To obtain the Adjoint representation of the $(1+1)-$dimensional Poincaré group, $ISO(1,1)$, consider its infinitesimal generators
$$\partial_t,\quad \partial_x\quad \textrm{and}\quad -x\partial_t-t\partial_x.$$

Let $g \in ISO(1,1)$ act on 
$$\alpha\partial_t+\beta\partial_x+\gamma(-x\partial_t-t\partial_x),$$
as follows
$$g\cdot(\alpha\partial_t+\beta\partial_x+\gamma(-x\partial_t-t\partial_x))=\alpha\partial_{\widetilde{t}}+\beta\partial_{\widetilde{x}}+\gamma(-\widetilde{x}\partial_{\widetilde{t}}-\widetilde{t}\partial_{\widetilde{x}}).$$
Substituting $\widetilde{t}$, $\widetilde{x}$, and the transformed differential operators $\partial_{\widetilde{t}}$ and $\partial_{\widetilde{x}}$ by \eqref{groupAction} and 
$$\begin{pmatrix}
\partial_{\widetilde{t}}\\
\partial_{\widetilde{x}}
\end{pmatrix}=\frac{\partial(\widetilde{t},\widetilde{x})}{\partial(t,x)}^{-T}\begin{pmatrix} \partial_t\\ \partial_x\end{pmatrix},$$
respectively, we obtain
\begin{equation}\label{AdgMotivating}
\begin{pmatrix} \alpha & \beta & \gamma\end{pmatrix}\underbrace{\begin{pmatrix}\cosh{\theta} & \sinh{\theta)} & 0\\
\sinh{\theta} & \cosh{\theta} & 0\\
-b\cosh{\theta}-a\sinh{\theta} & -b\sinh{\theta}-a\cosh{\theta} & 1\end{pmatrix}}_{\mathcal{A}d(g)}\begin{pmatrix}
\partial_t\\
\partial_x\\
-x\partial_t-t\partial_x
\end{pmatrix}.
\end{equation}
\end{example}

Having briefly summarised some moving frame concepts and symbolic invariant calculus, we proceed to the derivation of the generalised Euler-Lagrange system and its conserved quantities, as it bears many similarities to their invariantised versions, as we will see in section \ref{IGVP}.

\section{Generalised Variational Problem}

In this section, we use the symmetry method to derive the generalised Euler-Lagrange equations, as it simultaneously yields the associated generalised conserved quantities. 

Consider the generalised variational problem
\begin{equation}\label{varProblem1}
\frac{\mathrm{d}A}{\mathrm{d}\lambda}=L(\lambda,u^{(n)},A),
\end{equation}
where the Lagrangian is a smooth function of the independent variable, $\lambda$, the dependent variables and their derivatives up to order $n$, $u^{(n)}$, and the functional itself, $A$. Throughout this work, we consirder Herglotz Lagrangians to be affine in $A$, as in all classical and modern examples of Herglotz variational principle. Moreover, we assume that values of $u^\alpha_K$, for $\alpha=1,\dots,q$ and $K$ a multi-index of differentiation with $0\le |K|\le n-1$, are prescribed at an initial state $0$ and at a final state $\varlambda$. Note that for $A(\lambda)$ to be well-defined, $A(0)$ must have the same value, independently of $u(\lambda)$.

Since our results rely on the moving-frame method, which requires the group action to be free and regular on a domain $\mathcal{U}\subset J^n$, from now on, we shall suppose that the group acts freely and regularly. Thus, with this in mind, we introduce the following definition of a variational symmetry group for one-dimensional Herglotz variational problems.

\begin{definition}
Suppose $\mathcal{I}$ is a subdomain with closure $\bar{\mathcal{I}}\subset X$ and $u$ is a smooth function defined over $\mathcal{I}$. Then a local group of transformations $G$, which acts on the base space, $X\times U$, is a \textbf{variational symmetry group} of the functional defined by \eqref{varProblem1}, if whenever $A(\lambda)$ satisfies \eqref{varProblem1}, the transformed functional $\widetilde{A}(\widetilde{\lambda})$ satisfies
\begin{equation}
\frac{\mathrm{d}\widetilde{A}}{\mathrm{d}\widetilde{\lambda}}=L(\widetilde{\lambda},\widetilde{u^{(n)}},\widetilde{A}).
\end{equation}
\end{definition}

Let $g\in G$ act on the base space as follows:
\begin{equation}\label{groupAct}\begin{array}{l}
\widetilde{\lambda}=\lambda\\
\widetilde{u^\alpha}=\Phi^\alpha(\lambda,u;\varepsilon),\quad \alpha=1,\dots,q,
\end{array}
\end{equation}
where $\varepsilon$ is the parameter and $\Phi^\alpha(\lambda,u;0)=u^\alpha$. Note we assume that the group action leaves the independent variable invariant; if not, we can reparametise with respect to an invariant independent variable. Then, the infinitesimals are
$$\tau(\lambda,u)=\left.\frac{\mathrm{d}\widetilde{\lambda}}{\mathrm{d}\varepsilon}\right|_{\varepsilon=0}=0\quad \textrm{and}\quad \phi^\alpha(\lambda,u)=\left.\frac{\mathrm{d}\Phi^\alpha}{\mathrm{d}\varepsilon}\right|_{\varepsilon=0}.$$ 

Finally, suppose that $G$ is the variational symmetry group of the functional defined by \eqref{varProblem1}. Letting $g\in G$ act on \eqref{varProblem1}, followed by differentiation with respect to the group parameter and evaluation at $\varepsilon=0$ yields

\begin{align}\nonumber
&\quad\left.\frac{\mathrm{d}}{\mathrm{d}\varepsilon}\left(\frac{\mathrm{d}}{\mathrm{d}\widetilde{\lambda}}A(\widetilde{\lambda},\widetilde{u^{(n)}})\right)\right|_{\varepsilon=0}=\left.\frac{\mathrm{d}}{\mathrm{d}\varepsilon}L(\widetilde{\lambda},\widetilde{u^{(n)}},A(\widetilde{\lambda},\widetilde{u^{(n)}}))\right|_{\varepsilon=0}\\\nonumber
\Longleftrightarrow &\quad\frac{\mathrm{d}}{\mathrm{d}\lambda}\left.\frac{\mathrm{d}}{\mathrm{d}\varepsilon}A(\lambda,\widetilde{u^{(n)}})\right|_{\varepsilon=0}=\sum_{\alpha=1}^q\sum_K\frac{\partial L}{\partial u^\alpha_K}\frac{\mathrm{d}^{|K|}}{\mathrm{d}\lambda^{|K|}}\phi^\alpha+\frac{\partial L}{\partial A}\left.\frac{\mathrm{d}}{\mathrm{d}\varepsilon}A(\lambda,\widetilde{u^{(n)}})\right|_{\varepsilon=0}\\\label{diffwoutInt}
\Longleftrightarrow &\quad \frac{\mathrm{d}}{\mathrm{d}\lambda}\mathcal{A}(\lambda)-\frac{\partial L}{\partial A}\mathcal{A}(\lambda)=\sum_{\alpha=1}^q\sum_K\frac{\partial L}{\partial u^\alpha_K}\phi^\alpha_K,
\end{align}
where $K$ is a multi-index of differentiation with $0\le |K|\le n$ and
$$\mathcal{A}(\lambda)=\left.\frac{\mathrm{d}}{\mathrm{d}\varepsilon}A(\lambda,\widetilde{u^{(n)}})\right|_{\varepsilon=0}.$$

Multiplying \eqref{diffwoutInt} by an integrating factor $F=F(\lambda)$, such that
$$F\cdot\left(-\frac{\partial L}{\partial A}\right)=\frac{\mathrm{d}F}{\mathrm{d}\lambda}$$
we obtain
\begin{equation}\label{diffeqbefInt}
\frac{\mathrm{d}}{\mathrm{d}\lambda}\left(F(\lambda)\mathcal{A}(\lambda)\right)=F(\lambda)\sum_{\alpha=1}^q\sum_K\frac{\partial L}{\partial u^\alpha_K}\phi^\alpha_K.
\end{equation}

For convenience, we normalise the integrating factor by setting
$$\mathcal{F}(\lambda)=\frac{F(\lambda)}{F(0)}.$$ Hence, the solution to \eqref{diffeqbefInt} is given by
\begin{equation}\label{solucao}
\mathcal{F}(\varlambda)\mathcal{A}(\varlambda)-\mathcal{A}(0)=\int_0^\varlambda \mathcal{F}(\lambda)\left(\sum_{\alpha=1}^q\sum_K\frac{\partial L}{u^\alpha_K}\phi^\alpha_K\right)\,\mathrm{d}\lambda.
\end{equation}

Since we are looking for $u(\lambda)$ such that $A(\lambda)$ is stationary, i.e. an extremal, this means that the left-hand side of \eqref{solucao} must be zero. Thus,
\begin{eqnarray}\label{firstVariation}
\int_0^\varlambda \mathcal{F}\left(\lambda\right)
\sum_{\alpha=1}^q\sum_K \frac{\partial L}{\partial u^\alpha_K}\phi^\alpha_K\;\mathrm{d}\lambda &\kern-7pt=&\kern-7pt 0
\end{eqnarray}
where 
$$\mathcal{F}\left(\lambda\right)=\exp\left(-\int_0^{\lambda}\frac{\partial L}{\partial A}\mathrm{d}\theta\right).$$

Using integration by parts in \eqref{firstVariation}, we obtain
$$\begin{array}{ll}
&\displaystyle{\int^\varlambda_0 \sum_{\alpha=1}^q\left(\sum_{K}(-1)^{|K|}\frac{\mathrm{d}^{|K|}}{\mathrm{d}\lambda^{|K|}}\left(\mathcal{F}\left(\lambda\right)\frac{\partial L}{\partial u^\alpha_K}\right)\right)\phi^\alpha\,\mathrm{d}\lambda}\\
&\displaystyle{+\left.\left(\sum_{\alpha=1}^q\sum_K\sum_{k=1}^{|K|}(-1)^{k-1}\frac{\mathrm{d}^{k-1}}{\mathrm{d}\lambda^{k-1}}\left(\mathcal{F}\left(\lambda\right)\frac{\partial L}{\partial u^\alpha_K}\right)\frac{\mathrm{d}^{|K|-k}}{\mathrm{d}\lambda^{|K|-k}}\phi^\alpha\right)\right|_{\lambda=0}^{\lambda=\varlambda}=0,}
\end{array}$$
where the multi-index $K$ in the first line has order comprised between 0 and $n$, and in the second line between 1 and $n$.

Thus, the generalised Euler-Lagrange equations are
\begin{equation}
\mathsf{E}_{\mathcal{F}}^\alpha(L)=0, \quad \alpha=1,\dots,q,
\end{equation}
where $\mathsf{E}_{\mathcal{F}}^\alpha$ is the modified (weighted) Euler operator with respect to $u^\alpha$, similar to the one defined by \textcite{Olver1995},
$$\mathsf{E}_{\mathcal{F}}^\alpha=\sum_{K}(-1)^{|K|}\frac{\mathrm{d}^{|K|}}{\mathrm{d}\lambda^{|K|}}\left(\mathcal{F}\left(\lambda\right)\frac{\partial}{\partial u^\alpha_K}\right),\quad 0\le |K|\le n,$$
and the generalised conserved quantity along those solutions is
$$\sum_{\alpha=1}^q\sum_K\sum_{k=1}^{|K|}(-1)^{k-1}\frac{\mathrm{d}^{k-1}}{\mathrm{d}\lambda^{k-1}}\left(\mathcal{F}\left(\lambda\right)\frac{\partial L}{\partial u^\alpha_K}\right)\frac{\mathrm{d}^{|K|-k}}{\mathrm{d}\lambda^{|K|-k}}\phi^\alpha=\textrm{constant}.$$
The above generalised Euler-Lagrange equations and conserved quantity agree with those in the literature \parencite{GuentherGuentherGottschHerglotz,GeorgievaGuenther2002,SantosMartinsTorres2014,SantosMartinsTorres2015}, when we take $\tau=\mathrm{d}\widetilde{\lambda}/\mathrm{d}\varepsilon|_{\varepsilon=0}=0$.

Before we proceed to the derivation of the invariantised version of the generalised Euler-Lagrange equations and their conserved quantities, we will need to present a result regarding constraints on Herglotz generalised variational problems. This result will be needed for the application presented in section \ref{HerglotzApp}. 

Since a holonomic constraint can be viewed as a special case of a nonholonomic constraint, we will only prove here the Lagrange multiplier rule for nonholonomic constraints.

\begin{theorem}\textbf{(Lagrange Multiplier rule for nonholonomic constraints)}
Let 
\begin{equation}\label{varTheorem}
\frac{\mathrm{d}A}{\mathrm{d}\lambda}=L(\lambda,u^{(n)},A)
\end{equation}
be a Herglotz variational problem with prescribed initial and final states, where $u:\mathcal{I}\subset\mathbb{R}\rightarrow \mathbb{R}^q$ is sufficiently smooth. Furthermore, let 
\begin{equation}\label{constraints}
g\left(\lambda,u^{(n)}\right)=0
\end{equation}
be a system of nonholonomic constraints, where $g=(g^1,\dots,g^m):\mathcal{I}\times \mathbb{R}^{q\times n}\rightarrow \mathbb{R}^m$. Suppose there exists $v^i$ variables, for $i=1,\dots,m$, of type $u^\alpha_K$, with $\alpha=1,\dots,q$ and $K$ multi-indices of differentiation with $0\le |K|\le n$, such that
$$\det\left(\frac{\partial(g^1,\dots,g^m)}{\partial(v^1,\dots,v^m)}\right)\ne 0,$$
known as the regularity assumption. Moreover, suppose that $u$ satisfies the constraints \eqref{constraints} and is a weak constrained extremal of \eqref{varTheorem}. Then, there exists
$$\mu:\mathcal{I}\rightarrow\mathbb{R}^m$$
such that $u$ satisfies the generalised Euler-Lagrange equations of the augmented Lagrangian
$$L\left(\lambda,u^{(n)},A\right)-\mu(\lambda)\cdot g\left(\lambda,u^{(n)}\right).$$
\end{theorem}

\begin{proof}
Let $\phi$ be a smooth vector-valued function with compact support in $\mathcal{I}$, so that $\widetilde{u}=u+\varepsilon\phi$ also lies in the same function space as $u$. Since $u$ extremises \eqref{varTheorem}, as determined before, the first variation of $A(\lambda)$, $\mathcal{A}(\varlambda)$, is given by
$$\mathcal{A}(\varlambda)=\int^\varlambda_0 \exp\left(\int_{\lambda}^\varlambda\frac{\partial L}{\partial A}\,\mathrm{d}\theta\right)\left(\sum_{\alpha=1}^q\sum_K\frac{\partial L}{\partial u^\alpha_K}\phi^\alpha_K\right)\,\mathrm{d}\lambda.$$

Let $C=\left\{\left(\lambda,u^{(n)}\right): g\left(\lambda,u^{(n)}\right)=0\right\}$ be the constrained manifold. The admissible variations $\widetilde{u}$ must belong to $C$. Hence, differentiating the perturbed constraints \eqref{constraints} with respect to $\varepsilon$ and evaluating at $\varepsilon=0$ yields
\begin{equation}\label{linearized_constraint}
\sum_{\alpha=1}^q\sum_K\frac{\partial g^i}{\partial u^\alpha_K}\phi^\alpha_K=0, \quad i=1,\dots m.
\end{equation}

Let $\hat{\mathrm{d}}$ be the vertical differential, as defined in \textcite{Olver1993}. The vertical differential of $g^i$ is given by
$$\hat{\mathrm{d}}g^i=\sum_{\alpha=1}^q\sum_K\frac{\partial g^i}{\partial u^\alpha_K}\,\mathrm{d}u^\alpha_K, \quad \textrm{for} \quad i=1,\dots,m.$$
These vertical 1-forms can be viewed as linear maps from the space of vertical vector fields to the space of differential functions, i.e. $\hat{\mathrm{d}}g^i:V(TJ^n)\rightarrow \mathcal{P}$, such that
$$v=\sum_{\alpha=1}^q\sum_K\phi^\alpha_K\frac{\partial}{\partial u^\alpha_K}\mapsto \hat{\mathrm{d}}g^i(v)=\sum_{\alpha=1}^q\sum_K\frac{\partial g^i}{\partial u^\alpha_K}\phi^\alpha_K.$$

Thus, the linearised constraints \eqref{linearized_constraint} can be written as $\hat{\mathrm{d}}g^i(v)=0$ for $i=1,\dots,m$. These vertical 1-forms define both the vertical tangent space to $C$,
$$V(TC)=\left\{\left.v\in V(TJ^n)\right|\hat{\mathrm{d}}g^i(v)=0,\; \forall i\right\}$$
and the normal vertical cotangent subspace to $C$,
$$N^\ast_VC=\left<\mathrm{d}g^i|i=1,\dots,m\right>.$$

Since $\mathcal{A}(\varlambda)$ can be written as 
$\int_0^\varlambda\mathcal{F}(\lambda)\,\hat{\mathrm{d}}L(v)\,\mathrm{d}\lambda\,\displaystyle{=0}$, 
and we have assumed the regularity assumption, this means that $\hat{\mathrm{d}}L$ can be written as a linear combination of the $\hat{\mathrm{d}}g^i$, for $i=1,\dots,m$. Hence, 
$$\int_0^\varlambda\mathcal{F}(\lambda)\,\hat{\mathrm{d}}L(v)\,\mathrm{d}\lambda=\int_0^\varlambda\mathcal{F}(\lambda)\,\sum_{i=1}^m\mu^i(\lambda)\hat{\mathrm{d}}g^i(v)\,\mathrm{d}\lambda,$$
which simplifies to 
$$\int_0^\varlambda\mathcal{F}(\lambda)\hat{\mathrm{d}}\left(L-\sum_{i=1}^m\mu^i(\lambda)g^i\right)(v)\,\mathrm{d}\lambda =0.$$
\hfill $\Box$
\end{proof}

We are now ready to derive the invariantised version of the generalised Euler-Lagrange equations for one-dimensional Herglotz variational problem of order $n$.

\section{Invariantised Generalised Variational Problem}\label{IGVP}

Suppose the generalised variational problem \eqref{varProblem1} is invariant under a prolonged group action, thus we can rewrite it in terms of the generating differential invariants of that group action $I^\alpha_K$, which for simplicity we will refer to as $\kappa=(\kappa_1,\dots,\kappa_N)$, and their derivatives, $\mathcal{D}_J\,\kappa$. Since the Lagrangian in the generalised variational problem depends on derivatives of order at most $n$, the Lagrangian in terms of the generating invariants will be of order less than or equal to $n$ --- this will depend on the choice of normalisation equations. Thus, we have assumed that the invariantised Lagrangian is of order\footnote{The variational problem might not depend on all of the generating differential invariants of the group action. Here we present the most general case.} $n$.

Thus, rewriting \eqref{varProblem1} in terms of the generating differential invariants of that group action (which can easily be achieved using the Replacement Theorem) yields
\begin{equation}\label{invariantisedProb}
\mathcal{D}_\lambda A=\mathcal{L}(\lambda,\kappa,\mathcal{D}_J\kappa,A).
\end{equation}
Note that even though $\mathcal{D}_J=D_J$, we adopt $\mathcal{D}_J$ to clearly show that the problem has been invariantised.

Since 
$$\left.\frac{\mathrm{d}}{\mathrm{d}\varepsilon}A[\widetilde{u}]\right|_{\varepsilon=0} \quad \textrm{and} \quad \left.\frac{\mathrm{d}}{\mathrm{d}\nu}A[u]\right|_{u_\nu(\lambda,0)=\phi},$$
yield the same symbolic result, let us introduce a new dummy variable $\nu$ (which is invariant under the group action) to effect the variation. The introduction of this new independent variable, $\nu$, leads to $q$ new invariants, namely $I^\alpha_\nu=\left.g\cdot u^\alpha_\nu\right|_{g=\rho(\zede)}$, for $\alpha=1,\dots,q$ and $N$ syzygies
\begin{equation}\label{matrixSyzygies}
\mathcal{D}_\nu\begin{pmatrix}
\kappa_1\\
\vdots\\
\kappa_N
\end{pmatrix}=\mathcal{H}\begin{pmatrix}
I^1_\nu\\
\vdots\\
I^q_\nu
\end{pmatrix},
\end{equation}
written in matrix form, where $\mathcal{H}$ is a $N\times q$ matrix of operators that depend on $\mathcal{D}_\lambda^k$, for $k\in \mathbb{N}$, $\kappa_j$, for $j=1,\dots,N$, and their derivatives.

As $u^\alpha_\nu(\lambda,0)=\phi^\alpha$, for $\alpha=1,\dots,q$, the $I^\alpha_\nu$ will play the same role in the derivation of the invariantised version of the generalised Euler-Lagrange equations and their conserved quantities, as $\phi^\alpha$ did in the derivation of the generalised Euler-Lagrange equations. So proceeding as for the derivation of the generalised Euler-Lagrange equations, i.e. differentiating \eqref{invariantisedProb} with respect to $\nu$, we obtain
\begin{align}\nonumber
&\mathcal{D}_\nu\mathcal{D}_\lambda A = \mathcal{D}_\nu\,\mathcal{L}(\lambda,\kappa,\mathcal{D}_J\kappa,A) \qquad\Longleftrightarrow\\\label{IntFactorInv}
&\mathcal{D}_\lambda\mathcal{D}_\nu A = \sum_{j=1}^N\sum_{k=0}^n\frac{\partial \mathcal{L}}{\partial \mathcal{D}_\lambda^k\kappa_j}\mathcal{D}_\lambda^k\mathcal{D}_\nu\kappa_j+\frac{\partial \mathcal{L}}{\partial A}\mathcal{D}_\nu A,
\end{align}
since $[\mathcal{D}_\lambda,\mathcal{D}\nu]=0$. Now, set the term $\mathcal{D}_\nu A$ as $\mathcal{A}=\mathcal{A}(\lambda)$.

As before, we multiply \eqref{IntFactorInv} by an integrating factor $F(\lambda)$, such that
$$F\cdot\left(-\frac{\partial \mathcal{L}}{\partial A}\right)=\mathcal{D}_\lambda F,$$
resulting in the solution
\begin{equation}\label{solInv}
F(\varlambda)\mathcal{A}(\varlambda)-F(0)\mathcal{A}(0)=\int^\varlambda_0 F\left(\lambda\right)\sum_{j=1}^N\sum_{k=0}^n\frac{\partial \mathcal{L}}{\partial \mathcal{D}_{\lambda}^k\kappa_j}\mathcal{D}_{\lambda}^k\mathcal{D}_\nu\kappa_j\,\mathrm{d}\lambda,
\end{equation}
where 
$$F(\lambda)=F(0)\exp\left(-\int^\lambda_0\frac{\partial \mathcal{L}}{\partial A}\,\mathrm{d}\theta\right).$$

Since we are looking for solutions that leave $A(\lambda)$ stationary, this implies that the left-hand side of \eqref{solInv} must be zero, i.e.
\begin{equation}\label{intparts1}
\int^\varlambda_0 \mathcal{F}\left(\lambda\right)\sum_{j=1}^N\sum_{k=0}^n\frac{\partial \mathcal{L}}{\partial \mathcal{D}^k_\lambda\kappa_j}\mathcal{D}_{\lambda}^k\mathcal{D}_\nu\kappa_j\,\mathrm{d}\lambda=0
\end{equation}
where, as previously, we have normalised the integrating factor, i.e.
\begin{equation}\label{eqIntFac2}
\mathcal{F}(\lambda)=\frac{F(\lambda)}{F(0)}=\exp\left(-\int^\lambda_0\frac{\partial \mathcal{L}}{\partial A}\,\mathrm{d}\theta\right).
\end{equation}

Performing a first set of integration by parts in \eqref{intparts1}, we obtain
\begin{equation}\label{integralBefSys}
\begin{array}{l}
\displaystyle{\int^\varlambda_0\sum_{j=1}^N\sum_{k=0}^n(-1)^k\mathcal{D}_\lambda^k\left(\mathcal{F}(\lambda)\frac{\partial \mathcal{L}}{\partial \mathcal{D}_{\lambda}^k\kappa_j}\right)\mathcal{D}_\nu\kappa_j\,\mathrm{d}{\scriptstyle{T}}}\\[10pt]
\quad \displaystyle{+\left.\left(\sum_{j=1}^N\sum_{k=1}^n\sum_{\ell=1}^k(-1)^{\ell-1}\mathcal{D}_{\lambda}^{\ell-1}\left(\mathcal{F}(\lambda)\frac{\partial \mathcal{L}}{\partial \mathcal{D}_{\lambda}^k\kappa_j}\right)\mathcal{D}_{\lambda}^{k-\ell}\mathcal{D}_\nu\kappa_j\right)\right|_{\lambda=0}^{\lambda=\varlambda}}=0.
\end{array}
\end{equation}

Substituting each $\mathcal{D}_\nu\kappa_j$ in the integral part of \eqref{integralBefSys} by the respective syzygies in \eqref{matrixSyzygies}, we obtain
\begin{equation}\label{integralBefInt}
\begin{array}{l}
\displaystyle{\int^\varlambda_0\sum_{j=1}^N\mathsf{E}_{\mathcal{F}}^j(\mathcal{L})\sum_{\alpha=1}^q(\mathcal{H})_{j,\alpha}I^\alpha_\nu\,\mathrm{d}\lambda}\\[10pt]
\quad \displaystyle{+\left.\left(\sum_{j=1}^N\sum_{k=1}^n\sum_{\ell=1}^k(-1)^{\ell-1}\mathcal{D}_{\lambda}^{\ell-1}\left(\mathcal{F}(\lambda)\frac{\partial \mathcal{L}}{\partial \mathcal{D}_{\lambda}^k\kappa_j}\right)\mathcal{D}_\lambda^{k-\ell}\mathcal{D}_\nu\kappa_j\right)\right|_{\lambda=0}^{\lambda=\varlambda}}=0.
\end{array}
\end{equation}
where $\mathsf{E}_{\mathcal{F}}^j$ is the modified Euler operator with respect to $\kappa_j$. As mentioned earlier, the $I^\alpha_\nu$ play the same role as $\phi^\alpha$, for $\alpha=1,\dots,q$. Thus, to obtain the invariantised version of the generalised Euler-Lagrange equations, we need to isolate the $I^\alpha_\nu$ in the integral part of \eqref{integralBefInt} and for that we perform a second set of integration by parts, which yields
\begin{equation}\label{intparts2}
\begin{array}{l}
\displaystyle{\int^\varlambda_0\sum_{\alpha=1}^q\sum_{j=1}^N(\mathcal{H})^\ast_{j,\alpha}\left(\mathsf{E}_{\mathcal{F}}^j(\mathcal{L})\right)I^\alpha_\nu\,\mathrm{d}\lambda}\\[10pt]
\quad \displaystyle{+\left.\left(\sum_{j=1}^N\sum_{k=1}^n\sum_{\ell=1}^k(-1)^{\ell-1}\mathcal{D}_{\lambda}^{\ell-1}\left(\mathcal{F}(\lambda)\frac{\partial \mathcal{L}}{\partial \mathcal{D}_{\lambda}^k\kappa_j}\right)\mathcal{D}_{\lambda}^{k-\ell}\mathcal{D}_\nu\kappa_j\right)\right|_{\lambda=0}^{\lambda=\varlambda}}+\,\textrm{B.T.'s}=0.
\end{array}
\end{equation}
where $(\mathcal{H})^\ast_{j\alpha}$ is the adjoint operator of $(\mathcal{H})_{j,\alpha}$ and B. T.'s are boundary terms that we have picked up from this second set of integration by parts. 

Hence, from \eqref{intparts2}, we can read off the invariantised version of the generalised Euler-Lagrange equations, namely
$$\mathsf{E}^\alpha_{\mathcal{F}}(L)=\sum_{j=1}^N(\mathcal{H})^\ast_{j,\alpha}\left(\mathsf{E}_{\mathcal{F}}^j(\mathcal{L})\right)=0, \quad \alpha=1,\dots,q,$$
where $\mathsf{E}^\alpha_{\mathcal{F}}$ is the modified Euler operator with respect to $u^\alpha$ and $\mathsf{E}^j_{\mathcal{F}}$ is the modified Euler operator with respect to $\kappa_j$. This completes the proof of the following theorem.

\begin{theorem}
Consider the generalised variational problem
$$\frac{\mathrm{d}A}{\mathrm{d}\lambda}=L(\lambda,u^{(n)},A)$$
which is invariant under the prolonged group action of $G$ on the $n^{\textrm{th}}$ jet space, $J^n$. Let $\kappa=(\kappa_1,\dots,\kappa_N)$ be the generating differential invariants of that group action and $\mathcal{D}_J\kappa$ their derivatives with respect to $\lambda$. Furthermore, let the independent variable $\lambda$ be invariant under that group action. Then the invariantised version of the generalised Euler-Lagrange equations are
\begin{equation}\label{InvELeqns}
\mathsf{E}^\alpha_{\mathcal{F}}(L)=\sum_{j=1}^N(\mathcal{H})^\ast_{j,\alpha}\left(\mathsf{E}_{\mathcal{F}}^j(\mathcal{L})\right)=0, \quad \alpha=1,\dots,q,
\end{equation}
where $\mathsf{E}^\alpha_{\mathcal{F}}$ is the modified Euler operator with respect to $u^\alpha$, $\mathsf{E}^j_{\mathcal{F}}$ is the modified Euler operator with respect to $\kappa_j$ and $(\mathcal{H})^\ast_{j,\alpha}$ is the adjoint operator of $(\mathcal{H})_{j,\alpha}$ --- $\mathcal{H}$ is the matrix of operators in \eqref{matrixSyzygies}.
\end{theorem}

These invariantised generalised Euler-Lagrange equations reduce, in the classical one-dimen-sional case, to the invariantised Euler-Lagrange equations presented in \textcite{GoncalvesMansfield2011}. Before we proceed to demonstrate the structure of the generalised conserved quantities, let us now determine the invariantised generalised Euler-Lagrange equations from our motivational example.

\begin{example}
To determine the invariantised generalised Euler-Lagrange equations of the Herglotz problem \eqref{dampedVar} we must first rewrite it in terms of the generating invariants of the group action. For that we make use of the Replacement Theorem. Recall from Example \ref{exFrame} that the normalisation equations are: $\widetilde{t}=0$, $\widetilde{x}=0$ and $\widetilde{x_\lambda}=0$. Thus, $I^t=0$, $I^x=0$ and $I^x_1=0$ and the generating differential invariants are $I^t_1$ and $I^x_{11}$. Furthermore, remember that $I^t_1$ is the Minkowski norm, $\eta$, and instead of using $I^x_{11}$, we will use the Minkowski curvature, $\kappa$, which is equal to $I^x_{11}/(I^t_1)^2$.  Hence \eqref{dampedVar} becomes
$$\mathcal{D}_\lambda A=\frac{1}{2}\eta^2-\alpha A.$$

Using the formula for the invariantised generalised Euler-Lagrange equations \eqref{InvELeqns}, we realise we need to compute $\mathsf{E}^\eta_{\mathcal{F}}(\mathcal{L})$, $\mathsf{E}^\kappa_{\mathcal{F}}(\mathcal{L})$ and the adjoint matrix of operators\footnote{When calculating the adjoint matrix of operators, it is crucial to keep track of the boundary terms, as these will be needed to determine the generalised conserved quantities.} $\mathcal{H}$. Since $\mathcal{L}$ does not depend on $\kappa$, $\mathsf{E}^\kappa(\mathcal{L})=0$, which will imply that we do not need to bother calculating the second column of $\mathcal{H}^\ast$. 

Hence, the system of invariantised generalised Euler-Lagrange equations is
$$\left\{\begin{array}{l}
\mathsf{E}^t(L)=-\mathcal{D}_\lambda(\mathsf{E}^\eta_{\mathcal{F}}(\mathcal{L}))=-\mathcal{D}_\lambda\left(\mathrm{e}^{\alpha\lambda}\eta\right)=0,\\[10pt]
\mathsf{E}^x(L)=\kappa\eta\,\mathsf{E}^\eta_{\mathcal{F}}(\mathcal{L})=\mathrm{e}^{\alpha\lambda}\kappa\eta^2 =0.
\end{array}\right.$$
\end{example}

Note that all boundary terms in \eqref{intparts2}, obtained from the two sets of integration by parts performed earlier, produce terms that are linear combinations of the $I^\alpha_{K\nu}$, which can be written as
$$\left.\left(\mathcal{F}(\lambda)\sum_{\alpha=1}^q\sum_K C^\alpha_K I^\alpha_{K\nu}\right)\right|_{\lambda=0}^{\lambda=\varlambda},$$
where $C^\alpha_K$ are the coefficients of the $I^\alpha_{K\nu}$. To justify the appearance of $\mathcal{F}$ as a factor, we need the result in the following lemma.

\begin{lemma}\label{FRelation}
Let $\mathcal{F}(\lambda)$ be defined as in \eqref{eqIntFac2}. Furthermore, suppose that $G_\ell(I)$, with $\ell\in \mathbb{N}$, are functions in the cross section $\mathcal{K}$. Then 
$$\mathcal{D}_\lambda^\ell(\mathcal{F}(\lambda))=\mathcal{F}(\lambda)G_\ell(I),$$
for all $\ell \in \mathbb{N}$.
\end{lemma}

\begin{proof}
We shall prove this by induction. 

\vspace{0.2cm}
\noindent \textit{Induction base}\\
\indent For $\ell=1$ 
$$\mathcal{D}_\lambda(\mathcal{F}(\lambda))=\mathcal{D}_\lambda\left(\frac{F(\lambda)}{F(0)}\right)=\frac{1}{F(0)}\mathcal{D}_\lambda(F(\lambda))=-\frac{F(\lambda)}{F(0)}\frac{\partial\mathcal{L}}{\partial A}=\mathcal{F}(\lambda)\left(-\frac{\partial\mathcal{L}}{\partial A}\right),$$
where $-\partial\mathcal{L}/\partial A$ is a function of the invariantised jet coordinates, which we shall denote as $G_1(I)$.

\vspace{0.2cm}
\noindent \textit{Induction step}\\
\indent Assume $\mathcal{D}_\lambda^\ell(\mathcal{F}(\lambda))=\mathcal{F}(\lambda)G_\ell(I)$. Then
$$\begin{array}{l}
\mathcal{D}_\lambda^{\ell+1}(\mathcal{F}(\lambda))=\mathcal{D}_\lambda\left(\mathcal{D}_\lambda^\ell(\mathcal{F}(\lambda))\right)=\mathcal{D}_\lambda\left(\mathcal{F}(\lambda)G_\ell(I)\right)=\mathcal{D}_\lambda(\mathcal{F}(\lambda))G_\ell(I)+\mathcal{F}(\lambda)\mathcal{D}_\lambda(G_\ell(I))\\[10pt]
\displaystyle{\quad =-\mathcal{F}(\lambda)\frac{\partial \mathcal{L}}{\partial A}G_\ell(I)+\mathcal{F}(\lambda)\mathcal{D}_\lambda(G_\ell(I))=\mathcal{F}(\lambda)\left(-\frac{\partial \mathcal{L}}{\partial A}G_\ell(I)+\mathcal{D}_\lambda(G_\ell(I))\right).}
\end{array}$$
Since 
$$-\frac{\partial \mathcal{L}}{\partial A}G_\ell(I)+\mathcal{D}_\lambda(G_\ell(I))$$
is a function of the invariantised jet coordinates, without loss of generality we can name it $G_{\ell+1}(I)$. 

\vspace{0.2cm}
Thus, we have proved by induction that $\mathcal{D}_\lambda^\ell(\mathcal{F}(\lambda))=\mathcal{F}(\lambda)G_\ell(I)$. \hfill $\Box$
\end{proof}

As all terms in the boundary terms involve either $\mathcal{F}(\lambda)$ or $\mathcal{D}_\lambda^\ell(\mathcal{F}(\lambda))$, then it is easy to show, using Lemma \ref{FRelation}, that $\mathcal{F}$ factors out. We are now ready to show the structure behind the generalised conserved quantities arising from the Noether-type theorem for one-dimensional generalised variational problems of order $n$, as stated and proved in \textcite{SantosMartinsTorres2015}.

\begin{theorem}
Consider the generalised variational problem
\begin{equation}\label{FunctTheorem}
\frac{\mathrm{d}A}{\mathrm{d}\lambda}=L(\lambda,u^{(n)},A),
\end{equation}
where the Lagrangian is a smooth function of the independent variable, $\lambda$, the dependent variables and their derivatives up to order $n$, $u^{(n)}$, and the functional itself, $A$. Let \eqref{FunctTheorem} be invariant under the prolonged group action of $G$ on $J^n$, which leaves $\lambda$ and $A$ invariant. Consider $\kappa=(\kappa_1,\dots,\kappa_N)$ to be the generating differential invariants and $\mathcal{D}_J\kappa$ their derivatives. Thus, \eqref{FunctTheorem} can be written as
\begin{equation}\label{FunctTheoremInv}
\mathcal{D}_\lambda A=\mathcal{L}(\lambda,\kappa,\mathcal{D}_J\kappa,A).
\end{equation}

Introducing a dummy independent variable, $\nu$ --- which is also invariant under the group action --- to effect the variation, differentiating \eqref{FunctTheoremInv} with respect to $\nu$, setting $\mathcal{D}_\nu A$ as $\mathcal{A}$, solving for $\mathcal{A}$ and then integrating by parts yields
$$\int_0^\varlambda\left[\sum_{\alpha=1}^q\mathsf{E}^\alpha_{\mathcal{F}}(L)I^\alpha_\nu+\mathcal{D}_{\lambda}\left(\sum_{\alpha=1}^q\sum_K\mathcal{F}(\lambda)I^\alpha_{K\nu}C^\alpha_K\right)\right]\,\mathrm{d}\lambda=0,$$
where this defines $\mathcal{C}=(C^\alpha_K)$ and 
$$\mathsf{E}^\alpha_{\mathcal{F}}(L)=\sum_{j=1}^N\left(\mathcal{H}\right)^\ast_{j,\alpha}\mathsf{E}^j_{\mathcal{F}}(\mathcal{L}).$$ 

Let $(\varepsilon_1,\dots,\varepsilon_r)$ be the coordinates of $G$ near the identity, $e$, and $\mathbf{w}_i$, for $i=1,\dots,r$, the associated infinitesimal generators. Furthermore, let $\mathcal{A}d(g)$ be the Adjoint representation of $G$ with respect to those infinitesimal vector fields. Set
$$\Omega(\zede)=(\omega^j_i),$$
as the matrix of infinitesimals, where
$$\omega^j_i=\left.\frac{\partial \zede_j}{\partial \varepsilon_i}\right|_{g=e},$$
are the infinitesimals of the prolonged group action. Let $\Omega(I)$ be the invariantised version of the above matrix. Then, the $r$ generalised conserved quantities, obtained via the Noether-type theorem, can be written as
\begin{equation}\label{structureCQ}
\mathcal{F}(\varlambda)\mathcal{A}d(\rho)^{-1}\boldsymbol{\upsilon}=\mathbf{c},
\end{equation}
where $\mathbf{c}$ is a constant vector and
\begin{equation}
\boldsymbol{\upsilon}=\Omega(I)\mathcal{C}
\end{equation}
is a vector of invariants.
\end{theorem}

\begin{proof}
We know that 
$$\left.\frac{\mathrm{d}}{\mathrm{d}\varepsilon}\right|_{\varepsilon=0}A(\lambda,\widetilde{u^{(n)}})) \quad \textrm{and} \quad \left.\frac{\mathrm{d}}{\mathrm{d}\nu}\right|_{u_\nu(\lambda,0)=\phi}A(\lambda,u^{(n)})$$
yield the same symbolic result. Thus, as shown previously in this section,
$$\mathcal{D}_\nu\mathcal{D}_\lambda A=\mathcal{D}_\nu \mathcal{L}(\lambda,\kappa,\mathcal{D}_J\kappa,A)$$
produces the boundary term
\begin{equation}\label{conservedQuant1}
\left.\left(\mathcal{F}(\lambda)\left(I^1_\nu \; I^1_{\lambda\nu} \dots I^1_{K_1\nu}\dots I^q_\nu \; I^q_{\lambda\nu} \dots I^q_{K_q\nu} \right)\mathcal{C}\right)\right|_{\lambda=0}^{\lambda=\varlambda}=0,
\end{equation}
where $K_i$, for $i=1,\dots,q$, are multi-indices of differentiation with respect to $\lambda$. By definition, $I^\alpha_{K\nu}$ is equal to
$$I^\alpha_{K\nu}=\widetilde{u^\alpha_{K\nu}}|_{g=\rho(\zede)}=\left.\frac{\mathrm{d}}{\mathrm{d}\nu}\widetilde{u^\alpha_K}\right|_{g=\rho(\zede)}.$$

Since by the chain rule
$$\frac{\mathrm{d}}{\mathrm{d}\nu}=\sum_{\alpha=1}^q\sum_Ku^\alpha_{K\nu}\frac{\partial}{\partial u^\alpha_K},$$
we have that 
\begin{align}\nonumber
&\left(I^1_\nu \; I^1_{\lambda\nu} \dots I^1_{K_1\nu}\dots I^q_\nu \; I^q_{\lambda\nu} \dots I^q_{K_q\nu} \right)\\\label{firstEqProof}
&\; =\left(u^1_\nu \; u^1_{\lambda\nu}\dots u^1_{K_1\nu} \dots u^q_\nu\; u^q_{\lambda\nu}\dots u^q_{K_q\nu}\right)\left.\left(\frac{\partial(\widetilde{u^1},\widetilde{u^1_\lambda},\dots,\widetilde{u^1_{K_1}},\dots,\widetilde{u^q},\widetilde{u^q_\lambda},\dots,\widetilde{u^q_{K_q}})}{\partial(u^1,u^1_\lambda,\dots,u^1_{K_1},\dots,u^q,u^q_\lambda,\dots,u^q_{K_q})}\right)^T\right|_{g=\rho(\zede)}.
\end{align}

For simplicity, let $\left(u^1 \; u^1_{\lambda}\dots u^1_{K_1} \dots u^q\; u^q_{\lambda}\dots u^q_{K_q}\right)$ be written as $\left(\zede_1\;\dots\;\zede_\ell\right)$, and hence \eqref{firstEqProof} can be simplified to
\begin{align}\nonumber
&\left(I^1_\nu \; I^1_{\lambda\nu} \dots I^1_{K_1\nu}\dots I^q_\nu \; I^q_{\lambda\nu} \dots I^q_{K_q\nu} \right)\\\label{secondEqProof}
&\qquad =\left(u^1_\nu \; u^1_{\lambda\nu}\dots u^1_{K_1\nu} \dots u^q_\nu\; u^q_{\lambda\nu}\dots u^q_{K_q\nu}\right)\left.\left(\frac{\partial\left(\widetilde{\zede_1},\dots,\widetilde{\zede_\ell}\right)}{\partial\left(\zede_1,\dots,\zede_\ell\right)}\right)^T\right|_{g=\rho(\zede)}.
\end{align}

Substituting the vector $\left(I^1_\nu \; I^1_{\lambda\nu} \dots I^1_{K_1\nu}\dots I^q_\nu \; I^q_{\lambda\nu} \dots I^q_{K_q\nu} \right)$ by \eqref{secondEqProof} in \eqref{conservedQuant1} yields
\begin{equation}\label{ConservedQuant2}
\left.\mathcal{F}(\lambda)\left(\left(u^1_\nu \; u^1_{\lambda\nu}\dots u^1_{K_1\nu} \dots u^q_\nu\; u^q_{\lambda\nu}\dots u^q_{K_q\nu}\right)\left.\left(\frac{\partial\left(\widetilde{\zede_1},\dots,\widetilde{\zede_\ell}\right)}{\partial\left(\zede_1,\dots,\zede_\ell\right)}\right)^T\right|_{g=\rho(\zede)}\mathcal{C}\right)\right|_{\lambda=0}^{\lambda=\varlambda}=0.
\end{equation}

Letting 
$$\left.\frac{\partial \widetilde{\zede_j}}{\partial \nu}\right|_{g=e}=\frac{\partial \zede_j}{\partial \nu}=\omega^j_i=\left.\frac{\partial \widetilde{\zede_j}}{\partial \varepsilon_i}\right|_{g=e}, \quad \textrm{for}\quad i=1,\dots,r,\quad j=1,\dots,\ell$$
the vector $\left(u^1_\nu \; u^1_{\lambda\nu}\dots u^1_{K_1\nu} \dots u^q_\nu\; u^q_{\lambda\nu}\dots u^q_{K_q\nu}\right)$ can be substituted by every single row of the matrix of infinitesimals $\Omega(\zede)$, as defined in \eqref{MatrixInfs}. Then Equation \eqref{ConservedQuant2} becomes
$$\left.\mathcal{F}(\lambda)\left(\boldsymbol{\omega}_i(\zede)^T\left.\left(\frac{\partial\left(\widetilde{\zede_1},\dots,\widetilde{\zede_\ell}\right)}{\partial\left(\zede_1,\dots,\zede_\ell\right)}\right)^T\right|_{g=\rho(\zede)}\mathcal{C}\right)\right|_{\lambda=0}^{\lambda=\varlambda}=0, \quad \textrm{for}\quad i=1,\dots,r.$$
These $r$ equations can be written in matrix form as
$$\left.\mathcal{F}(\lambda)\left(\Omega(\zede)\left.\left(\frac{\partial\left(\widetilde{\zede_1},\dots,\widetilde{\zede_\ell}\right)}{\partial\left(\zede_1,\dots,\zede_\ell\right)}\right)^T\right|_{g=\rho(\zede)}\mathcal{C}\right)\right|_{\lambda=0}^{\lambda=\varlambda}=\mathbf{0}.$$

Using Equation \eqref{essentialPiece} we can rewrite the above matrix form as
$$\left.\mathcal{F}(\lambda)\left(\left.\mathcal{A}d(g)^{-1}\Omega(\widetilde{\zede})\right|_{g=\rho(\zede)}\mathcal{C}\right)\right|_{\lambda=0}^{\lambda=\varlambda}=\mathbf{0},$$
which simplifies to
$$\mathcal{F}(\varlambda)\mathcal{A}d(\rho)^{-1}\Omega(I)\mathcal{C}=\mathbf{c},$$
which completes the proof.\hfill $\Box$
\end{proof}

When the Lagrangian does not depend on $A$, the structure of these generalised conservation laws reduces to the structure of the classical conservation laws presented in \textcite{GoncalvesMansfield2011}. Finally, let us determine the conserved quantities from our motivating example.

\begin{example}
Consider once more the invariantised Herglotz problem
$$\mathcal{D}_\lambda A = \dfrac{1}{2}\eta^2-\alpha A.$$
To obtain the generalised conserved quantities we need to compute $\mathcal{F}(\varlambda)$, $\mathcal{A}d(\rho)^{-1}$, $\Omega(I)$ and $\mathcal{C}$. As $\mathcal{F}(\varlambda)$ and $\mathcal{A}d(g)$ have already been computed previously, only the latter two are left to determine together with the evaluation of $\mathcal{A}d(g)$, \eqref{AdgMotivating}, at $g=\rho$, \eqref{frame_param}. We shall start with $\mathcal{C}$, as it will inform us how to compute $\Omega(I)$. The vector $\mathcal{C}$ corresponds to the coefficients of the $I^\alpha_{K\nu}$ in the boundary terms. As there was only one integration by parts, namely
$$\mathsf{E}^\eta_{\mathcal{F}}(\mathcal{L})\mathcal{D}_\lambda I^t_2,$$
the only boundary term obtained was
$$\mathcal{D}_\lambda\left(\mathsf{E}^\eta_{\mathcal{F}}(\mathcal{L})I^t_2\right).$$
Thus, $\mathcal{C}=\begin{pmatrix}\mathsf{E}^\eta_{\mathcal{F}}(\mathcal{L}) \end{pmatrix}$. The invariants involving $\nu$ in the boundary terms inform us which infinitesimals must compose $\Omega(I)$. As there was only $I^t_2$ in the boundary terms, $\Omega(I)$ will only be composed of the invariantised infinitesimals regarding $t$, namely $\phi^{t}_{,a}(I)$, $\phi^{t}_{,b}(I)$ and $\phi^{t}_{,\theta}(I)$. Hence,
$$\Omega(I)=\begin{pmatrix}
1\\
0\\
0
\end{pmatrix}.$$
Multiplying $\Omega(I)$ and $\mathcal{C}$ together yields the vector of invariants, $\boldsymbol{\upsilon}$. Putting all the pieces together we get the generalised conserved quantities
$$\mathrm{e}^{\alpha\lambda}\begin{pmatrix}
\dfrac{t_\lambda}{\eta} & -\dfrac{x_\lambda}{\eta} & 0\\
-\dfrac{x_\lambda}{\eta} & \dfrac{t_\lambda}{\eta} & 0\\
\dfrac{tx_\lambda-xt_\lambda}{\eta} & \dfrac{xx_\lambda-tt_\lambda}{\eta} & 1
\end{pmatrix}\begin{pmatrix}
\eta\\
0\\
0
\end{pmatrix}=\begin{pmatrix}
c_1\\
c_2\\
c_3
\end{pmatrix}.$$
\end{example}

In the remainder of the present paper, we shall look at Herglotz variational problems that are invariant under the action of the restricted Lorentz group $SO^+(1,2)$.

\section{Herglotz problems invariant under $\boldsymbol{SO^+(1,2)}$}\label{HerglotzApp}

Consider the $SO^+(1,2)$ group action given by
\begin{equation}\label{SO2action}
\begin{pmatrix}
\widetilde{t}\\
\widetilde{x}\\
\widetilde{y}
\end{pmatrix}=
\begin{pmatrix}
1 & 0 & 0\\
0 & \cos{\theta} & -\sin{\theta}\\
0 & \sin{\theta} & \cos{\theta}
\end{pmatrix}
\begin{pmatrix}
\cosh{\varsigma} & 0 & -\sinh{\varsigma}\\
0 & 1 & 0\\
-\sinh{\varsigma} & 0 & \cosh{\varsigma}
\end{pmatrix}
\begin{pmatrix}
\cosh{\varphi} & -\sinh{\varphi} & 0\\
-\sinh{\phi} & \cosh{\phi} & 0\\
0 & 0 & 1
\end{pmatrix}
\begin{pmatrix}
t\\
x\\
y
\end{pmatrix}.
\end{equation}

Furthermore, let us consider the cross section $\mathcal{K}$ to be defined by $\widetilde{x}=0$, $\widetilde{y}=0$ and $\widetilde{y_\lambda}=0$. Solving for these normalisation equations we obtain the frame
\begin{align*}
\theta=\arctan\left({\frac{y(tt_\lambda-xx_\lambda)-y_\lambda(t^2-x^2)}{\sqrt{t^2-x^2-y^2}(tx_\lambda-xt_\lambda)}}\right), \; \varsigma=\textrm{arctanh}\left({\frac{y}{\sqrt{t^2-x^2}}}\right), \; \varphi = \textrm{arctanh}\left({\frac{x}{t}}\right),
\end{align*}
in parametric form.

The lowest order differential invariants are
\begin{align*}
&I^t=\sqrt{t^2-x^2-y^2},\quad I^x_1=\frac{\sqrt{(yt_\lambda-ty_\lambda)^2+(tx_\lambda-xt_\lambda)^2-(xy_\lambda-yx_\lambda)^2}}{\sqrt{t^2-x^2-y^2}},\\
&I^y_{11}=\frac{t(x_\lambda y_{\lambda\lambda}-y_\lambda x_{\lambda\lambda})-x(t_\lambda y_{\lambda\lambda}-y_\lambda t_{\lambda\lambda})+y(t_\lambda x_{\lambda\lambda}-x_\lambda t_{\lambda\lambda})}{\sqrt{(yt_\lambda-ty_\lambda)^2+(tx_\lambda-xt_\lambda)^2-(xy_\lambda-yx_\lambda)^2}},
\end{align*}
where $I^t$ represents the proper time interval between the origin and the event $(t,x,y)$, $I^x_1$ the timelike relativistic areal velocity and $I^y_{11}$ the orbital curvature of the worldline. For simplicity we will denote $I^t$, $I^x_1$ and $I^y_{11}$ by $\tau$, $h$ and $\chi$, respectively.

Suppose we have a Herglotz variational problem that is invariant under the $SO^+(1,2)$ action defined in \eqref{SO2action}. This means the variational problem can be written in terms of the generating differential invariants $\tau$, $h$ and $\chi$ and their derivatives, say
$$\mathcal{D}_\lambda A = \mathcal{L}(\tau,\tau_\lambda,\tau_{\lambda,\lambda},h,h_\lambda,\chi,A).$$ 

Let us fix parametrisation as Minkowski arc length by imposing
$$t_\lambda^2-x_\lambda^2-y_\lambda^2=1.$$
Hence, by Theorem \ref{varTheorem} let us consider instead the constrained Herglotz variational problem
\begin{equation}\label{InvHerglotzProb}
\mathcal{D}_\lambda A=\mathcal{L}(\tau,\tau_\lambda,\tau_{\lambda\lambda},h,h_\lambda,\chi,A)-\mu(\lambda)\mathcal{R}(\tau_\lambda,h),
\end{equation}
where $\mathcal{R}(\tau_\lambda,h)$ corresponds to the invariantised Minkowski arc length constraint.

To effect the variation, we introduce a new dummy variable $\nu$ that is invariant under the $SO^+(1,2)$ action. From this, a set of syzygies arises, namely,
\begin{align*}
&\mathcal{D}_\nu \tau=I^t_2,\\[10pt]
&\mathcal{D}_\nu h=\frac{h}{\tau}I^t_2+\left(-\frac{\tau_\lambda}{\tau}+\mathcal{D}_\lambda\right)I^x_2-\frac{\chi}{h}I^y_2,\\[10pt]\nonumber
&\mathcal{D}_\nu \chi=\frac{\chi}{\tau}I^t_2+\left(\frac{\chi_\lambda}{h}-\frac{2\chi h_\lambda}{h^2}-\frac{\chi\tau_\lambda}{h\tau}+\frac{2\chi}{h}\mathcal{D}_\lambda\right)I^x_2\\[10pt]
&\qquad+\left(\frac{\tau_\lambda h_\lambda}{\tau h}+\frac{\tau_\lambda^2}{\tau^2}-\frac{\tau_{\lambda\lambda}}{\tau}-\frac{h^2}{\tau^2}-\frac{\chi^2}{h^2}-\frac{\mathcal{D}_\lambda(\tau h)}{\tau h}\mathcal{D}_\lambda+\mathcal{D}_\lambda^2\right)I^y_2.
\end{align*}

Using the results obtained in section \ref{IGVP}, that is the formula for the invariantised generalised Euler-Lagrange equations, \eqref{InvELeqns}, and the formula for the generalised conserved quantities, \eqref{structureCQ}, we obtain, respectively,
\begin{align*}
&\mathsf{E}^t(L)=\mathsf{E}^\tau_{\mathcal{F}}(\mathcal{L})+\mathcal{D}_\lambda\left(\mathcal{F}\mu\frac{\partial R}{\partial \tau_\lambda}\right)+\frac{h}{\tau}\left(\mathsf{E}^h_{\mathcal{F}}(\mathcal{L})-\mathcal{F}\mu\frac{\partial R}{\partial h}\right)+\frac{\chi}{\tau}\mathsf{E}^\chi_{\mathcal{F}}(\mathcal{L})=0,\\
&\mathsf{E}^x(L)=-\mathcal{D}_\lambda\left(\tau\left(\mathsf{E}^h_{\mathcal{F}}(\mathcal{L})-\mathcal{F}\mu\frac{\partial R}{\partial h}\right)\right)-\frac{2\chi\tau}{h}\mathcal{D}_\lambda\left(\mathsf{E}^\chi_{\mathcal{F}}(\mathcal{L})\right)-\frac{\mathcal{D}_\lambda(\chi\tau)}{h}\mathsf{E}^\chi_{\mathcal{F}}(\mathcal{L})=0,\\\nonumber
&\mathsf{E}^y(L)=-\frac{\chi}{h}\left(\mathsf{E}^h_{\mathcal{F}}(\mathcal{L})-\mathcal{F}\mu\frac{\partial R}{\partial h}\right)+\mathcal{D}_\lambda^2(\mathsf{E}^\chi_{\mathcal{F}}(\mathcal{L}))+\frac{\mathcal{D}_\lambda(\tau h)}{\tau h}\mathcal{D}_\lambda(\mathsf{E}^\chi_{\mathcal{F}}(\mathcal{L}))\\
&\qquad\qquad\qquad\qquad\qquad\quad+\left(\mathcal{D}_\lambda\left(\frac{h_\lambda}{h}\right)+\frac{\tau_\lambda h_\lambda}{\tau h}-\frac{h^2}{\tau^2}-\frac{\chi^2}{h^2}\right)\mathsf{E}^\chi_{\mathcal{F}}(\mathcal{L})=0,
\end{align*} 
and 
$$\begin{pmatrix}
\dfrac{t}{\tau} & \dfrac{t_\lambda}{h}-\dfrac{t\tau_\lambda}{\tau h} & -\dfrac{xy_\lambda-yx_\lambda}{\tau h}\\[10pt]
-\dfrac{x}{\tau} & \dfrac{x_\lambda}{h}-\dfrac{x\tau_\lambda}{\tau h} & -\dfrac{yt_\lambda-ty_\lambda}{\tau h}\\[10pt]
\dfrac{y}{\tau} & \dfrac{y\tau_\lambda}{\tau h}-\dfrac{y_\lambda}{h} & \dfrac{tx_\lambda-xt_\lambda}{\tau h}
\end{pmatrix}\begin{pmatrix}
h\mathsf{E}^\chi_{\mathcal{F}}(\mathcal{L})\\
\tau\mathcal{D}_\lambda\left(\mathsf{E}^\chi_{\mathcal{F}}(\mathcal{L})\right)+\dfrac{\tau_\lambda h_\lambda}{h}\mathsf{E}^\chi_{\mathcal{F}}(\mathcal{L})\\
-\tau\mathsf{E}^h_{\mathcal{F}}(\mathcal{L})+\tau\mathcal{F}\mu\dfrac{\partial R}{\partial h}-\dfrac{\chi\tau}{h}\mathsf{E}^\chi_{\mathcal{F}}(\mathcal{L})
\end{pmatrix}=\begin{pmatrix}
c_1\\
c_2\\
c_3
\end{pmatrix},
$$
where $c_i$, for $i=1,2,3$, are constants of integration. Note that $\mathcal{F}$ has been left inside the vector of invariants, as this yields a tidier arrangement of the terms.

In order to allow further analyses, from this point onward we shall only consider generalised variational problems that do not depend on $\chi$, as the syzygy involving $\chi$ is more complicated. In this situation, the invariantised generalised Euler-Lagrange equations reduce to
\begin{align*}
&\mathsf{E}^t(L)=\mathsf{E}^\tau_{\mathcal{F}}(\mathcal{L})+\mathcal{D}_\lambda\left(\mathcal{F}\mu\frac{\partial R}{\partial \tau_\lambda}\right)+\frac{h}{\tau}\left(\mathsf{E}^h_{\mathcal{F}}(\mathcal{L})-\mathcal{F}\mu\frac{\partial R}{\partial h}\right)=0,\\
&\mathsf{E}^x(L)=-\mathcal{D}_\lambda\left(\tau\left(\mathsf{E}^h_{\mathcal{F}}(\mathcal{L})-\mathcal{F}\mu\frac{\partial R}{\partial h}\right)\right)=0,\\
&\mathsf{E}^y(L)=-\frac{\chi}{h}\left(\mathsf{E}^h_{\mathcal{F}}(\mathcal{L})-\mathcal{F}\mu\frac{\partial R}{\partial h}\right)=0,
\end{align*}
and the conserved quantities simplify to
$$\begin{pmatrix}
\dfrac{t}{\tau} & \dfrac{t_\lambda}{h}-\dfrac{t\tau_\lambda}{\tau h} & -\dfrac{xy_\lambda-yx_\lambda}{\tau h}\\[10pt]
-\dfrac{x}{\tau} & \dfrac{x_\lambda}{h}-\dfrac{x\tau_\lambda}{\tau h} & -\dfrac{yt_\lambda-ty_\lambda}{\tau h}\\[10pt]
\dfrac{y}{\tau} & \dfrac{y\tau_\lambda}{\tau h}-\dfrac{y_\lambda}{h} & \dfrac{tx_\lambda-xt_\lambda}{\tau h}
\end{pmatrix}\begin{pmatrix}
0\\
0\\
-\tau\mathsf{E}^h_{\mathcal{F}}(\mathcal{L})+\tau\mathcal{F}\mu\dfrac{\partial R}{\partial h}
\end{pmatrix}=\begin{pmatrix}
c_1\\
c_2\\
c_3
\end{pmatrix}.$$

Solving $\mathsf{E}^y(L)=0$ yields
$$\chi=0 \qquad\vee\qquad \mu=\frac{1}{\mathcal{F}\partial R/\partial h}\;\mathsf{E}^h_{\mathcal{F}}(\mathcal{L}), \qquad h\ne 0.$$

Thus, this invariantised generalised Euler-Lagrange system has solutions split into two classes, according to whether $\chi$ vanishes or not. So for the case where $h\ne 0$ and $\chi\ne 0$, we must take 
$$\mu=\frac{1}{\mathcal{F}\partial R/\partial h}\;\mathsf{E}^h_{\mathcal{F}}(\mathcal{L}).$$ 
Substituting $\mu$ into $\mathsf{E}^t(L)=0$ produces
$$\mathcal{D}_\lambda\left(\frac{\tau_\lambda}{h}\;\mathsf{E}^h_{\mathcal{F}}(\mathcal{L})\right)=\mathsf{E}^\tau_{\mathcal{F}}(\mathcal{L}),$$
while $\mathsf{E}^x(L)=0$ is automatically satisfied. Furthermore, substituting $\mu$ in the generalised conserved quantities, yields a null vector of invariants, which reduces them to trivial generalised conservation laws. Thus, for $h\ne 0$ and $\chi\ne 0$, to find the extremals we must solve
$$\mathcal{D}_\lambda\left(\frac{\tau_\lambda}{h}\;\mathsf{E}^h_{\mathcal{F}}(\mathcal{L})\right)=\mathsf{E}^\tau_{\mathcal{F}}(\mathcal{L}),$$
together with the constraint $\tau_\lambda^2-h^2=1$, for $\tau$ and $h$. In principle, once $\tau$ and $h$ are known, the original variables, $t$, $x$ and $y$, can be reconstructed by acting on the zeroth invariants, $(\tau,0,0)$, with $\rho(\zede)^{-1}$. For further information on invariant reconstruction and moving frames, see \textcite{Mansfield2010}.

Considering now the case where $h\ne 0$ and $\chi=0$, we use $\mathsf{E}^x(L)=0$ to solve for the Lagrange multiplier $\mu$. Hence, we obtain
$$\mu=\frac{\tau\mathsf{E}^h_{\mathcal{F}}(\mathcal{L})-k_1}{\tau\mathcal{F}\partial R/\partial h},$$ 
where $k_1$ is a constant of integration.

Substituting $\mu$ in $\mathsf{E}^t(L)=0$ and in the vector of invariants of the generalised conserved quantities, simplifies them to
$$\mathcal{D}_\lambda\left(\frac{\tau_\lambda}{h}\mathsf{E}^h_{\mathcal{F}}(\mathcal{L})-\frac{k_1\tau_\lambda}{\tau h}\right)=\mathsf{E}^\tau_{\mathcal{F}}(\mathcal{L})+\frac{k_1h}{\tau^2}$$
and
$$\boldsymbol{\upsilon}=(0,0,-k_1)^T,$$
respectively. 

This vector of invariants reduces the generalised conservation laws to $\mathcal{A}d(\rho)\mathbf{c}=\boldsymbol{\upsilon}$, i.e.
\begin{align}\label{cons1}
&\dfrac{t}{\tau}c_1+\dfrac{x}{\tau}c_2-\dfrac{y}{\tau}c_3=0,\\[10pt]\label{cons2}
&\left(\dfrac{t_\lambda}{h}-\dfrac{t\tau_\lambda}{\tau h}\right)c_1+\left(\dfrac{x_\lambda}{h}-\dfrac{x\tau_\lambda}{\tau h}\right)c_2-\left(\dfrac{y_\lambda}{h}-\dfrac{y\tau_\lambda}{\tau h}\right)c_3=0,\\[10pt]\label{cons3}
& \left(\dfrac{xy_\lambda-yx_\lambda}{\tau h}\right)c_1-\left(\dfrac{yt_\lambda-ty_\lambda}{\tau h}\right)c_2-\left(\dfrac{tx_\lambda-xt_\lambda}{\tau h}\right)c_3=-k_1.
\end{align}

Solving \eqref{cons1} for $t$ gives
\begin{equation}\label{solforT}
t=c_3y-c_2x,
\end{equation}
which forces \eqref{cons2} to be automatically satisfied and simplifies \eqref{cons3} to
\begin{equation}\label{diffEqforXY}
\dfrac{xy_\lambda-yx_\lambda}{c_1\tau h}=\dfrac{1}{k_1},
\end{equation}
where we have used the equality $c_1^2-c_2^2-c_3^2=-k_1^2$, which comes from the relation between the generalised conserved quantities. Equations \eqref{solforT} and \eqref{diffEqforXY}, together with $t_\lambda^2-x_\lambda^2-y_\lambda^2=1$, allow to solve for $t$, $x$ and $y$, once we have solved $\mathsf{E}^t(L)=0$ and the invariantised constraint for $\tau$ and $h$. In theory, this might be simpler than rebuilding the solution by acting on $(\tau,0,0)$ with $\rho(\zede)^{-1}$, as this implies solving an extra three differential equations of higher-order.

This illustrates that, for Herglotz variational problems that are invariant under the restricted Lorentz group $SO^+(1,2)$, structural knowledge from the invariantised generalised Euler-Lagrange system and its conserved quantities can significantly simplify the computation of the extremals.

\section{Conclusion}

This work has established an intrinsic geometric and algebraic framework for one-dimensional higher-order Herglotz variational problems, including invariantised Euler–Lagrange equations and semi‑invariant conserved quantities. These results clarify the structural mechanisms underlying nonconservative variational systems and extend the geometric picture beyond the classical conservative setting. They also make explicit how symmetry acts within the contact structure, thereby integrating invariance and nonconservative effects into a single geometric picture.

Future work could extend the results presented here to generalised variational problems with several independent variables, whether or not these transform under the group action.

\vspace{1cm}

\noindent\textsf{\textbf{Acknowlegments}}
\begin{sloppypar}\noindent Torres is supported by CIDMA under the Portuguese Foundation for Science and Technology (FCT), Grant UID/04106/2025 \url{https://doi.org/10.54499/UID/04106/2025}). Frederico received financial support from Funda\c{c}\~{a}o Cearense de Apoio ao Desenvolvimento Cient\'{i}fico e Tecnol\'{o}gico (FUNCAP), Agency Process No. BP6-0241-
00049.01.00/25.
\end{sloppypar}

\vspace{1cm}
\noindent\textsf{\textbf{Authors contributions}}\\ 
\noindent Gast\~{a}o Frederico proposed the idea. T\^{a}nia Gon\c{c}alves developed the theory and proved the results presented in this paper. All authors discussed and contributed to the final version of this manuscript.

\vspace{1cm}
\noindent\textsf{\textbf{Data availability}}\\
\noindent No datasets were generated or analysed during the current study.

\section*{Declaration}
\noindent\textsf{\textbf{Conflict of interest}}\\ 
\noindent The authors declare that there are no conflicts of interest associated with the present manuscript.

\newpage
\printbibliography

\end{document}